\documentclass[11pt]{article}
\usepackage{latexsym,amssymb,amsmath,a4wide,theorem}

\usepackage{esint}

\numberwithin{equation}{section}
\numberwithin{figure}{section}
\newtheorem{theorem}{Theorem}[section]
\newenvironment{proof}{{ \it Proof:\quad}}{\hfill $\blacksquare$\par}
\newtheorem{lemma}{Lemma}[section]
\newtheorem{remark}{Remark}[section]
\newtheorem{proposition}{Proposition}[section]
\newtheorem{corollary}{Corollary}[section]
\newtheorem{definition}{Definition}[section]

\newcommand{\f}{\frac}
\newcommand{\lf}{\left}
\newcommand{\rg}{\right}

\title{Stability of the phase separation state for compressible Navier-Stokes/Allen-Cahn system}

\author{Yazhou Chen\thanks{Department of Mathematics, College of Mathematics and Physics, Beijing University
		of Chemical Technology, Beijing 100029,
		P R  China, (chenyz@mail.buct.edu.cn).}
	\and Hakho Hong\thanks{Institute of Mathematics, State Academy of Sciences, Pyongyang,
		D P R Korea (hhhong@star-co.net.kp and hhong@amss.ac.cn).}
	\and Xiaoding Shi\thanks{Corresponding author. Department of Mathematics,  College of Mathematics and Physics, Beijing University of Chemical Technology, Beijing 100029,
		P R  China (shixd@mail.buct.edu.cn) }
}

\date{}

\begin{document}
	\maketitle

\begin{abstract}
This paper is concerned with the large time behavior of the Cauchy problem for Navier-Stokes/Allen-Cahn system  describing the interface motion of immiscible two-phase flow in 3-D. The existence and uniqueness of global solutions and the stability of the phase separation state is proved under the small initial perturbations. Moreover, the optimal time decay rates are obtained for higher-order spatial derivatives of density, velocity and phase. Our results implies that if the immiscible two-phase flow is initially located near the phase separation state, then under small perturbation conditions, the solution exists globally and decays algebraically to the complete separation state of the two-phase flow, that is,  there will be no interface fracture, vacuum, shock wave, mass concentration at any time, and the interface thickness tends to zero as the time $t\rightarrow+\infty$.
	
	\vspace{.20cm}\noindent\textbf{MSC 2020:}
35B40, 35B65, 35L65, 76N05, 76N10, 76T10.
	
	\vspace{.20cm} \noindent\textbf{Keywords:} Navier-Stokes/Allen-Cahn	system, existence, uniqueness, large time behavior, decay rate.
\end{abstract}

%\tableofcontents

\section{Introduction}
\setcounter{equation}{0}
\hspace{2em}Two-phase flows or multi-phase flows are important in many industrial applications,
for instance, in aerospace, chemical engineering, micro-technology and so on. They have
attracted studies from many   engineers, geophysicists and astrophysicists. In this paper, we study a diffusive interface model which is coupled with the  compressible barotropic Navier-Stokes equations and Allen-Cahn equation, called the compressible  barotropic Navier-Stokes/Allen-Cahn system as follows (see \cite{HMR 12}):
\begin{equation}\label{h11}\left\{\begin{aligned}
&\displaystyle\rho_{t}+{\rm div}(\rho \mathbf{u})= 0, \\
&\displaystyle\lf(\rho \mathbf{u}\rg)_t+{\rm div}(\rho \mathbf{u}\otimes \mathbf{u})
={\rm div}\mathbb{T},\\
&\displaystyle\lf(\rho \phi\rg)_t+{\rm div}(\rho \mathbf{u}\phi)
=-\mu,\\
&\displaystyle\rho\mu=\rho\f{\partial f}{\partial\phi}-{\rm div}\lf(\rho\f{\partial f}{\partial\nabla\phi}\rg),
\end{aligned}\right.
\end{equation} where $\mathbf{x}=(x_1,x_2,x_3)\in  \mathbb{R}^3$ represents the spatial variable, $t>0$ represents the time variable, $\mathrm{div}$ and $\nabla$ are the divergence operator and gradient operator respectively. $\rho=\rho(\mathbf{x},t), \mathbf{u}=\mathbf{u}(\mathbf{x},t)=(u_1, u_2, u_3)(\mathbf{x},t)$  and $\phi=\phi(\mathbf{x},t)$ denote the density,
the velocity and the concentration difference of the mixture fluids respectively, and while $\mu$ is the chemical potential, $f$ is the phase-phase interfacial free energy density, here we consider its
common form as following (see Lowengrub-Truskinovsky \cite{LT 98},  Heida-M\'alek-Rajagopal \cite{HMR 12})
\begin{equation}\label{h12}
f(\rho, \phi, \nabla\phi)=\f{1}{4\epsilon}\lf(\phi^2-1\rg)^2+\f{\epsilon}{2\rho}|\nabla\phi|^2,
\end{equation}where $\epsilon>0$ is the thickness of the interface between the phases. The Cauchy stress-tensor $\mathbb{T}$
is represented by
\begin{equation}\label{h13}
\mathbb{T}=\nu \lf(\nabla \mathbf{u}+\nabla^\top \mathbf{u}\rg)+\lambda {\rm div}\mathbf{u}\mathbb{I}-P\mathbb{I}-\rho\nabla\phi\otimes\f{\partial f}{\partial\nabla\phi},
\end{equation}where $\mathbb{I}$ is the unit matrix, $\top$ represents the transpose of a matrix, and $\nu, \lambda$ are viscosity coefficients satisfying
\begin{equation}\label{h14}
\nu>0, \quad \lambda+\f{2}{3}\nu\geq 0.
\end{equation}
The total pressure is
\begin{equation}\label{h15}
P=p(\rho)-\f{\epsilon}{2}|\nabla\phi|^2,
\end{equation} which implies the sum of fluid pressure and capillary pressure. In this paper, we assume that the pressure $p$ in \eqref{h15} is the smooth function of $\rho$, and it holds that
\begin{equation}\label{h110}
p'(\rho)>0 \quad  \text{for any}\,\,\,\rho>0.
\end{equation}
Substituting \eqref{h12}, \eqref{h13} and \eqref{h15}  into \eqref{h11}, then \eqref{h11} is simplified  as
\begin{equation}\label{h17}\left\{\begin{aligned}
&\displaystyle\rho_{t}+{\rm div}(\rho \mathbf{u})= 0, \\
&\displaystyle\begin{aligned}\rho \mathbf{u}_t+\rho (\mathbf{u}\cdot\nabla) \mathbf{u}&+\nabla p(\rho)
=\nu\Delta\mathbf{u}+(\nu+\lambda)\nabla{\rm div}\mathbf{u}\\
&\displaystyle-\epsilon{\rm div}\big( \nabla\phi\otimes\nabla\phi -\f{1}{2}|\nabla\phi|^2\mathbb{I}\big),\end{aligned}\\
&\rho \phi_t+\rho \mathbf{u}\cdot\nabla\phi
=-\mu, \\
&\displaystyle\rho\mu=\f{\rho}{\epsilon}\lf(\phi^3-\phi\rg)-\epsilon\Delta\phi.
\end{aligned}\right.
\end{equation}

\begin{remark}
The  phase function $\phi$ is introduced to identify the two fluids ($\{ \mathbf{x}:\phi(\mathbf{x},t) = 1\}$ is occupied by fluid 1 and $\{ \mathbf{x}:\phi(\mathbf{x},t) = -1\}$ by fluid 2). Physically, the function $\phi$ and the total density are constructed as follows: for compressible immiscible two-phase flow, Take any volume element $V$ in the flow, $M_i$ is assumed to be the mass of the components in the representative material volume $V$, $\phi_i=\frac{\rho_i}{\rho}$ the mass concentration, $\rho_i=\frac{M_i}{V}$  the apparent mass density of the fluid $i~(i=1,2)$. The total density is given by $\rho=\rho_1+\rho_2$, and the difference of the two components for the fluid mixture $\phi=\phi_1-\phi_2$. We also call $\phi$ the phase function or phase field. Obviously, $\phi$  describes the distribution of the interface.
\end{remark}

\vspace{0.2cm} During the past decade, the mathematical study for the Navier-Stokes/Allen-Cahn system has been extensively
studied.
For a diffuse interface model of two viscous
fluids which lead to the  incompressible Navier-Stokes/Allen-Cahn system,
a first result  on existence of
axisymmetric solutions was obtained by Xu-Zhao-Liu \cite{XZL 10}. In the case with matched density, Zhao-Guo-Huang \cite{ZGH 11} studied the existence of weak solution in 3D, well-posedness of
strong solution in 2D and the vanishing
viscosity limit. Also, we refer to \cite{GG 10-1, GG 10-2, M 12} for  the asymptotic behavior and attractors. Recently, Favre-Schimperna \cite{FS 19} showed the existence of weak solution in 3D and well-posedness of
strong solution in 2D on a incompressible model with inertial effects. For the incompressible model with different densities, Li-Ding-Huang \cite{LDM 16} established a blow-up criterion
for  strong solutions and  Li-Huang \cite{LM 18} studied the existence and uniqueness of local strong solutions in 3D case. Moreover, by using an energetic variational approach, Jiang-Li-Liu \cite{JLL 17} derived a different model of Navier-Stokes/Allen-Cahn, then proved the existence of weak solutions in 3D,
the well-posedness of strong solutions in 2D, and studied the long time behavior of the strong solutions.

For the initial boundary value problem of the compressible barotropic model \eqref{h17}  in 3-D bounded domain,
Feireisl-Petzeltov\'a-Rocca-Schimperna \cite{FPRS 10} developed a rigorous existence theory  based on the concept of weak solution for the compressible Navier-Stokes system
introduced by Lions for the adiabatic exponent of pressure $\gamma>6$. This result
was recently extended to $\gamma>2$ by Chen-Wen-Zhu \cite{CWZ 19} . Freist\"uler \cite{F 14} showed that the possibility of traveling waves corresponding to phase boundaries arises during a phase transition at a critical temperature under natural assumptions.
For the problem obtained by linearizing the system \eqref{h17} around a traveling waves, Kotschote \cite{K 17} proved the results on local well-posedness and a detailed description of the point and essential spectrum.
For the Cauchy problem of the 3-D compressible Navier-Stokes/Allen-Cahn system, Zhao \cite{Zhaoxp} studied the global well-posedness and time-decay rates of solutions via a refined pure energy method, in which the equilibrium of the concentration difference $\phi$ takes the value $0$.
For the 1-D problem of the system \eqref{h17} in bounded interval,  Ding-Li-Luo \cite{DLL 13} proved the existence and uniqueness of global classical solution, the existence
of weak solutions and the existence of unique strong solution for initial
data $\rho_0$ without vacuum states.
Later,  Chen-Guo \cite{CG 17}  established the global existence and uniqueness of
strong and classical solutions  by using the
energy estimates, when the initial vacuum is allowed. Recently, Ding-Li-Tang \cite{DLT 19}  obtained the similar result as in \cite{DLL 13} for the problem with free
boundary. Also,
Luo-Yin \cite{LY 18}  and Yin-Zhu \cite{YZ 19} proved the asymptotic stability toward the rarefaction wave of Cachy problem  in 1-D, and the combination of stationary solution and rarefaction wave for the initial boundary value problem in half line, respectively.

On the other hand, for the full compressible model, Kotschote \cite{K 12}  derived the more general model,  and  proved the existence and uniqueness of
local strong solutions on a problem with a mixed boundary condition in bounded domain.
Recently, Chen-He-Huang-Shi	\cite{CHFS 20-1, CHFS 20-2}  studied the global strong solutions for the 1-D Cauchy problem and initial boundary value problem, respectively.
	
%However, to the best of our knowledge, there seems not to be any result for global strong or smooth solution of the multi-dimensional compressible Navier-Stokes/Allen-Cahn system.

In this paper, we focus on the existence of the global smooth solutions and the stability of the phase separation state for the Cauchy problem of the system  \eqref{h17} in $\mathbb{R}^3$ with the initial condition
\begin{equation}\label{h19}
(\rho, \mathbf{u}, \phi)(\mathbf{x},0)=(\rho_0, \mathbf{u}_0, \phi_0)(\mathbf{x}),\quad \mathbf{x}\in \mathbb{R}^3.\end{equation}
As the space variable tends to infinity, we assume
\begin{equation}\label{hh19}
		\lim_{|\mathbf{x}|\rightarrow\infty}(\rho_0, \mathbf{u}_0)(\mathbf{x})= (\bar{\rho}, 0),\quad\text{and}\quad  \lim_{|\mathbf{x}|\rightarrow\infty} |\phi_0(\mathbf{x})|= 1,
\end{equation}
where  $\bar{\rho}>0$  is the given positive constant.

\begin{remark}
Compared with the results in \cite{Zhaoxp} about the global well-posedness and time-decay rates of solutions near the equilibrium state of the concentration difference of the mixture fluids $\phi=0$, we give the global solution and the similar time-decay rates near the equilibrium state of $|\phi|=1$, which represents the state of phase separation. More precisely, the condition $\eqref{hh19}$ can be used to describe immiscible two-phase flows in different phase fields when $\mathbf{x}\rightarrow\pm\infty$, so it can be used to describe the phenomenon of the phase separation in immiscible two-phase flow. From this perspective, the condition $\eqref{hh19}$ can be seen as an essential improvement on the condition (13) in \cite{LY 18} and (1.15) in \cite{YZ 19} which can be only described the disturbance near single-phase flow.
\end{remark}

Finally, we point out there are some mathematical results for the compressible Navier-Stokes-Cahn-Hilliard system which is another diffusive interface model  describing the motion of a
mixture of two compressible viscous
fluids (see  \cite{AF 08, KZ 15, CHMS 18}  and references therein).

\vspace{0.2cm}\noindent\textbf{Notation.} In this paper,
$L^p(\mathbb{R}^3)$ and $W_p^k(\mathbb{R}^3)$ denote the usual Lebesgue and Sobolev spaces on $\mathbb{R}^3$, with norms  $\|\cdot\|_{L^p}$ and $\|\cdot\|_{W_p^k}$, respectively. When $p=2$, we denote $W_p^k(\mathbb{R}^3)$ by $H^k(\mathbb{R}^3)$ with  the norm  $\|\cdot\|_{H^k}$ and $\|\cdot\|_{H^0}=\|\cdot\|$ will be used to denote the usual
$L^2-$norm.  The notation $\|(A_1,A_2, \cdots, A_l)\|_{H^k}$ means the summation of $\|A_i\|_{H^k}$ from $i=1$ to $i=l$. For an integer $m$, the symbol $\nabla^m$ denotes the summation of all terms $D^{\alpha}$ with the multi-index $\alpha$
satisfying $|\alpha|=m$. We use $C, c$ to denote the constants which are
independent of $\mathbf{x},t$
and may change from line to line. We also omit the spatial domain
$\mathbb{R}^3$ in integrals for convenience.
For $d\times d$-matrices $F, H$, denote $
F:H=\sum_{i,j=1}^{d}F_{ij}H_{ij}$, $|F|\equiv(F:F)^{1/2}$.  For vectors $a$ and $b$, we denote their tensor product by $a\otimes b:=(a_i
b_j)_{d\times d}$.  We will employ the notation $a\lesssim b$ to mean that $a\leq  Cb$ for a universal constant $C>0$ that only depends on the parameters coming from the problem.

\vspace{0.2cm}In order to establish the negative Sobolev estimates, we should review the following useful
results. To this end, let us first introduce the following necessary definition.
\begin{definition}\label{def21}
	For $s\in\mathbb{R}$, $\dot{H}^s(\mathbb{R}^3)$ is defined as the homogeneous Sobolev space of $f,$ with the following norm:
	$$
	\|f\|_{\dot{H}^{s}}=\|\Lambda^s f\|,
	$$where $\Lambda^s$ is defined by
	$$(\Lambda^s f)(\mathbf{x})=\f{1}{(2\pi)^3}\int_{\mathbb{R}^3}|\xi|^s \hat{f}(\xi)e^{2\pi i \mathbf{x}\cdot \xi}d\xi, $$ where $\hat{f}(\xi)$ is the Fourier transform of $f$,
	$\hat{f}(\xi)\overset{\mathrm{def}}= \displaystyle \int_{\mathbb{R}^3}f(\mathbf{x})e^{-2\pi i \mathbf{x}\cdot \xi} d\mathbf{x}.$
\end{definition}

Moreover, we need the following standard results which will be used extensively in our estimates.
\begin{lemma} \label{lem21} (Gagliardo-Nirenberg inequality, \cite{N 59} or \cite[Lemma A.1]{W 12})
	Let $l,s$ and $k$ be any real numbers satisfying $0\leq
	l,s<k$, and let  $p, r, q \in [1,\infty]$ and $\f{l}{k}\leq
	\theta\leq 1$ such that
	$$\frac{l}{3}-\frac{1}{p}=\lf(\frac{s}{3}-\frac{1}{r}\rg)(1-\theta)+\lf(\frac{k}{3}-\frac{1}{q}\rg)\theta.
	$$ Then, for any $u\in W_q^k(\mathbb{R}^3),$  we have
	\begin{equation}
	\label{h20}
	\|\nabla^l u\|_{L^p}\lesssim \|\nabla^s u\|_{L^r}^{1-\theta}\|\nabla^k
	u\|_{L^q}^{\theta}.
	\end{equation}	
\end{lemma}

\begin{lemma} \label{lem22}  (\cite[Lemma 2.5]{PG 16})  Let $f(\varphi)$ and $f(\sigma, w)$ be  smooth functions of $\varphi$ and $(\varphi, w)$, respectively, with bounded
	derivatives of any order, and $\|\varphi\|_{L^\infty(\mathbb{R}^3)}\lesssim 1$. Then for any integer $m\geq 1$, we have
	 	\begin{equation}
	 \label{hh20}\begin{aligned}
	 &	\|\nabla^m f(\varphi)\|_{L^p}\leq C \|\nabla^m \varphi\|_{L^p},
	 \\
	 &\|\nabla^m f(\varphi,w)\|_{L^p}\leq C  \|\nabla^m (\varphi,w)\|_{L^p},
	 \end{aligned}
\end{equation} for any $1\leq p\leq \infty$, where $C$	may depend on $f$ and $m$.	
\end{lemma}

%\begin{definition}\label{def21}
%	For $s\in\mathbb{R}$, $\dot{H}^s(\mathbb{R}^3)$ is defined as the homogeneous Sobolev space of $f,$ with the following norm:
%	$$
%	\|f\|_{\dot{H}^{s}}=\|\Lambda^s f\|,
%	$$where $\Lambda^s$ is defined by
%	$$(\Lambda^s f)(x)=\f{1}{(2\pi)^3}\int_{\mathbb{R}^3}|\xi|^s \hat{f}(\xi)e^{2\pi i x\cdot \xi}d\xi, $$ where $\hat{f}(\xi)$ is the Fourier transform of $f$ defined by
%	$\hat{f}(\xi)=\f{1}{(2\pi)^3}\int_{\mathbb{R}^3}f(x)e^{-2\pi i x\cdot \xi} dx.$
%\end{definition}
By the Parseval theorem and H\"older's inequality, it is easy to check  the following result (see \cite{W 12}).
\begin{lemma} \label{lem24}
	Let $s\geq 0$ and $l\geq 0$. Then, we have
	\begin{equation}
\label{hg24}\|\nabla^lf\|\lesssim\|\nabla^{l+1}f\|^{1-\theta}\|f\|^{\theta}_{\dot{H}^{-s}},\quad \text{with}\,\,\,\theta=\f{1}{l+s+1}.\end{equation}
\end{lemma}

If $s\in (0, 3)$, $\Lambda^{-s}g$ is the Riesz potential. Then, we have the following $L^p$ type inequality by the Hardy-Littlewood-Sobolev theorem (see \cite[pp. 119, Theorem 1]{S 70}):
\begin{lemma} \label{lem25}
	Let $0<s<3, 1<p<q<\infty$ and $\f{1}{q}+\f{s}{3}=\f{1}{p}$. Then, we have
		\begin{equation}
	\label{hg25}\|\Lambda^{-s}f\|_{L^q}\lesssim\|f\|_{L^p}.\end{equation}
\end{lemma}

\vspace{0.2cm}Now our main results are given as following:

\begin{theorem}  \label{theo11}
	Assume that \eqref{h14}, \eqref{h110}, and %for an integer $m\geq 3$,
\begin{equation}\label{rho}
 (\rho_0-\bar{\rho}, \mathbf{u}_0)\in
	H^{3}(\mathbb{R}^3),\quad\inf_{\mathbf{x}\in \mathbb{R}^3}\rho_0(\mathbf{x})>0,
\end{equation}
\begin{equation}\label{phi}
\nabla\phi_0\in
	H^{2}(\mathbb{R}^3),\quad\phi_0^2-1\in L^2(\mathbb{R}^3). %\quad\inf_{x\in \mathbb{R}^3}\phi_0^2>\frac{1}{3}.
\end{equation}
	Then, there exists a  positive constant  $\delta>0$  such that if
	\begin{equation}\label{g17}
	\big\|(\rho_0-\bar{\rho}, \mathbf{u}_0)\big\|_{H^{3}} +\| \nabla\phi_0\big\|_{H^{2}}+\|\phi_0^2-1\big\|\leq \delta,
	\end{equation} then
	the Cauchy problem \eqref{h17}-\eqref{hh19} admits a unique solution $(\rho,
	\mathbf{u}, \phi)$ on $[0,\infty)$ satisfying
	$$(\rho-\bar{\rho}, \mathbf{u}) \in C([0,\infty),
H^3(\mathbb{R}^3)),\ \phi^2-1\in C([0,\infty),L^{2}(\mathbb{R}^3)),\ \nabla\phi \in C([0,\infty),H^2(\mathbb{R}^3)),$$
	$$\nabla\rho\in L^2(0,\infty;H^2(\mathbb{R}^3)),\ \  \nabla\phi \in L^2(0,\infty;H^3(\mathbb{R}^3)),\ \ \nabla \mathbf{u}\in L^2(0,\infty;H^3(\mathbb{R}^3)),\ -1\leq\phi\leq1.
	$$
	and
	\begin{equation}\label{h18}
		\begin{aligned}
	&\sup_{t\in\mathbb{R}_+}\Big\{\big\|(\rho-\bar{\rho}, \mathbf{u})(t)\|^2_{H^3}+ \|\nabla \phi(t)\big\|^2_{H^2}+ \big\|\phi^2(t)-1\big\|^2\Big\}\\
&\quad+\int_0^{+\infty}\Big(\big\|\nabla\rho\big\|^2_{H^2}+\big\|\nabla\mathbf{u},\nabla\phi\big\|^2_{H^3} \Big)d\tau\\
	&\quad\leq C\Big(\big\|\rho_0-\bar{\rho}\big\|^2_{H^3}+\big\|\mathbf{u}_0\big\|^2_{H^3}+ \big\|\nabla \phi_0\big\|^2_{H^2}+\|\phi_0^2-1\|^2\Big).
	\end{aligned}\end{equation} where $C$ is the positive constant independent of
	$\mathbf{x}, t$ and  $ \delta$.

Moreover, if  $(\rho_0-\bar{\rho}, \mathbf{u}_0, \nabla\phi_0, \phi_0^2-1)\in \dot{H}^{-s}(\mathbb{R}^3)$ for some $s\in [0,\f{3}{2})$, then
	\begin{equation}
	\label{h124}
	\big\|(\rho-\bar{\rho}, \mathbf{u}, \nabla\phi, \phi^2-1)(t)\big\|_{\dot{H}^{-s}}^2\leq C_0,
	\end{equation} and
		\begin{equation}
	\label{hg124}
	\big\|\nabla^l(\rho-\bar{\rho}, \mathbf{u})(t)\big\|_{H^{3-l}}^2+	\big\|\nabla^l(\nabla\phi, \phi^2-1)(t)\big\|_{H^{2-l}}^2\leq C_0(1+t)^{-(l+s)},
	\end{equation}
for $l=0,1,2$, where $\dot{H}^{-s}(\mathbb{R}^3)$ denotes the homogeneous negative Sobolev space.
\end{theorem}

\vspace{0.2cm}If the initial value has a higher regularity, the above Theorem \ref{theo11} can be generalized as follows.
\begin{corollary}  \label{corol1}
 assuming that $N\geq3$, $ (\rho_0-\bar{\rho}, \mathbf{u}_0)\in	H^{3}(\mathbb{R}^3)$,  $\nabla\phi_0\in H^{2}(\mathbb{R}^3),\quad\phi_0^2-1\in L^2(\mathbb{R}^3)$, and all the other assumptions in the theorem 1.2 remain unchanged, then, the regularity of the solution which obtained in Theorem 1.2 can also be improved correspondingly, i.e.
$$(\rho-\bar{\rho}, \mathbf{u}) \in C([0,\infty),H^N(\mathbb{R}^3)),\ \phi^2-1\in C([0,\infty),L^{2}(\mathbb{R}^3)),\ \nabla\phi \in C([0,\infty),H^{N-1}(\mathbb{R}^3)),$$
	$$\nabla\rho\in L^2(0,\infty;H^{N-1}(\mathbb{R}^3)),\ \  \nabla\phi \in L^2(0,\infty;H^N(\mathbb{R}^3)),\ \ \nabla \mathbf{u}\in L^2(0,\infty;H^N(\mathbb{R}^3)),	$$
and
$$\big\|\nabla^l(\rho-\bar{\rho}, \mathbf{u})(t)\big\|_{H^{N-l}}^2+	\big\|\nabla^l(\nabla\phi, \phi^2-1)(t)\big\|_{H^{N-1-l}}^2\leq C_0(1+t)^{-(l+s)},\ l=0,\cdots,N-1.$$
\end{corollary}

\vspace{0.2cm} By Lemma \ref{lem25}, we obtain that for $p\in(1,2]$, $L^p(\mathbb{R}^3)\subset \dot{H}^{-s}(\mathbb{R}^3)$ with $s=3\lf(\f{1}{p}-\f{1}{2}\rg)\in [0, \f{3}{2})$. Then by Theorem \ref{theo11}, we have the following corollary of the usual $L^p-L^2$ type of optimal decay results:
\begin{corollary}  \label{corol2}
Under the assumptions of Corollary \ref{corol1} except that we replace the $\dot{H}^{-s}$ assumption by that $(\rho_0-\bar{\rho}, \mathbf{u}_0, \nabla\phi_0, \phi_0^2-1)\in L^p(\mathbb{R}^3)$ for some $p\in(1,2]$, then the following decay results hold:
	\begin{equation}
\label{h126}
	\big\|\nabla^l(\rho-\bar{\rho}, \mathbf{u})(t)\big\|_{H^{N-l}}^2+\big\|\nabla^l(\nabla\phi, \phi^2-1)(t)\big\|_{H^{N-1-l}}^2\lesssim  (1+t)^{-l-3\lf(\f{1}{p}-\f{1}{2}\rg)},
\end{equation}
 for $l=0,1,\cdots,N-1$.
\end{corollary}

The followings are several remarks for Theorem \ref{theo11} and Corollary \ref{corol1}-\ref{corol2}.

\begin{remark}\label{r1}
1) For the global existence of the solution, we only assume that the $H^3$-norm
of initial data is small, while the higher-order Sobolev norms can be arbitrarily large, as in several works for compressible fluid flows (see \cite{W 12, TZ 11, GTY 16}).

2) Notice that	the decay rate \eqref{h126}  of global solution is optimal in the sense that it coincides with
the decay rate of solutions to the linearized system
$$\left\{\begin{aligned}
&\sigma_t+\bar{\rho}{\rm div} \mathbf{u}=0,\\
&\mathbf{u}_t-\nu\bar{\rho}^{-1}\Delta\mathbf{u}-(\nu+\lambda)\bar{\rho}^{-1}\nabla{\rm
	div}
\mathbf{u}+\f{p'(\bar{\rho})}{\bar{\rho}}\nabla\sigma+\frac{\epsilon}{\bar\rho}\nabla\phi\Delta\phi=0,\\
& \phi_t-\f{\epsilon}{\bar{\rho}^2}\Delta\phi=0.\end{aligned}\right.
$$(cf. \cite[Remark 1.2]{GTY 16}).
\end{remark}

\begin{remark}

1) Note that we obtain the $L^2$ optimal decay rate of the higher-order spatial derivatives of the solution in  Corollary \ref{corol2}. Then the general optimal $L^q (q> 2)$ decay rates of the solution follow by the Gagliardo-Nirenberg inequality (see Lemma \ref{lem21}). For instance, it follows from \eqref{h126} that
$$\big\|(\rho-\bar{\rho}, \mathbf{u}, \nabla\phi, \phi^2-1)(t)\big\|_{L^q}^2\lesssim (1+t)^{-3\lf(\f{1}{p}-\f{1}{q}\rg)},\ \mathrm{for}\ 2<q\leq \infty.$$

2) The constraint
$s<\f{3}{2} $ in Theorem \ref{theo11}  comes from applying Lemma \ref{lem25}  to estimate the nonlinear terms when doing
the negative Sobolev estimates. This
in turn restricts $p>1$ in Corollary \ref{corol2} by our method. Note that the nonlinear estimates would not work for $s\geq\f{3}{2}$.

3) Using  the Green function, a priori estimates and the Fourier splitting method developed by Schonbek \cite{S 95}, we can prove Corollary \ref{corol2} with $p=1$ by the same lines as in \cite[Section 4]{GTY 16}.

\end{remark}

\begin{remark}
By using the Sobolev embedding theorem, it can be obtained that the initial data $\inf\limits_{\mathbf{x}\in \mathbb{R}^3}\phi_0^2>\frac{1}{3}$ when $\delta$ in \eqref{g17} is small enough. From a physical point of view, we know that $\phi_0=\pm\frac{\sqrt{3}}3$ are the critical points of the  phase-phase interfacial free energy density $f$, $\inf\limits_{\mathbf{x}\in \mathbb{R}^3}\phi_0^2>\frac{1}{3}$ means that the initial phase field is in the stable region, that is, in a state near phase separation. Theorem \ref{theo11} implies that if the immiscible two-phase flow is initially located near the phase separation state, then under small perturbation conditions, the solution exists globally and decays algebraically to the complete separation state of the two-phase flow, that is, the interface thickness tends to zero as the time $t\rightarrow+\infty$.
	
\end{remark}
As we know,  for the Cauchy problem of the compressible fluid flows in $\mathbb{R}^3$, there exist  several works to study the global well-posedness and algebraic decay estimates of smooth solutions: we refer to  \cite{GW 11} for Navier-Stokes system, \cite{W 12} for Navier-Stokes-Poisson system, \cite{TZ 11} for Navier-Stokes-Kortweg system and \cite{GTY 16} for nematic liquid crystal flows.  To prove Theorem \ref{theo11}, we will use the energy method  developed by \cite{GW 11}, which relies essentially on  the following  two main steps:

Step 1. Energy estimates at $l-$th level:
\begin{equation}
\label{h117}\begin{aligned}
\f{d}{dt}\mathcal{E}_l(t)&+\big\|\nabla ^{l+1} \rho(t)\big\|^2_{H^{3-l-1}}+ \big\|\nabla^{l+1}\mathbf{u}(t)\big\|^2_{H^{3-l}}+ \big\|\nabla^{l+1} \phi(t)\big\|^2_{H^{3-l}}\leq 0,
\end{aligned}
\end{equation}for any $0\leq l\leq 3$, where
$$\mathcal{E}_l(t) \backsimeq \big\|\nabla^l(\rho-\bar{\rho}, \mathbf{u})(t)\big\|^2_{H^{3-l}}+\big\|\nabla^{l+1} \phi(t)\big\|^2_{H^{3-1-l}}+\big\|\phi^2-1\big\|^2,$$
and $A\backsimeq   B$ means $CA\leq B\leq \f{1}{C}A$ for a generic constant $C>0.$

Step 2. Negative Sobolev norm estimate:
\begin{equation}
\label{h118}
\f{d}{dt}\mathcal{E}_{-s}(t)\lesssim \Big(\big\|\nabla\rho\big\|_{H^2}^2+
\big\|\nabla(\mathbf{u}, \nabla \phi)\big\|_{H^1}^2+\big\|\nabla(\phi^2-1)\big\|^2\Big)\mathcal{E}_{-s}(t),\end{equation}for $0 < s \leq \f{1}{2},$ and \begin{equation}
\label{hh118}
\f{d}{dt}\mathcal{E}_{-s}(t)\lesssim \big\|(\rho-\bar{\rho}, \mathbf{u}, \phi^2-1, \nabla\phi)\big\|^{s-\f{1}{2}}\Big(
\big\|\nabla(\rho, \mathbf{u},  \nabla\phi)\big\|_{H^1}+\big\|\nabla(\phi^2-1)\big\|\Big)^{\f{5}{2}-s}\mathcal{E}_{-s}(t),\end{equation}for $\f{1}{2}< s<\f{3}{2},$  where
$$\mathcal{E}_{-s}=\big\|\Lambda^{-s} (\rho-\bar{\rho})\big\|^2+\big\|\Lambda^{-s} \mathbf{u}\big\|^2+\big\|\Lambda^{-s} \nabla\phi\big\|^2+\big\|\Lambda^{-s} (\phi^2-1)\big\|^2. $$
If we prove \eqref{h117}, then it is easy to show that there exits a  solution of the system \eqref{h17}-\eqref{hh19} satisfying \eqref{h18} by the continuation argument of local solution (Subsection \ref{subs1}). Moreover, by using \eqref{h118} or \eqref{hh118}, and  Lemma \ref{lem24} and the following estimate on  differential inequality
\begin{equation}
	\label{h119}\f{df(t)}{dt}+c_0(f(t))^{1+\f{1}{l+s}}\leq 0\Rightarrow f(t)\leq \lf(f(0)^{-\f{1}{l+s}}+\f{c_0t}{l+s}\rg)^{-(l+s)}\lesssim (1+t)^{-(l+s)},\end{equation}
we can get the decay estimates \eqref{hg124}  (Subsection \ref{subs2}). Therefore, the estimates \eqref{h117}-\eqref{hh118} are essential in the proof of Theorem \ref{theo11}.

 Here, we briefly review some  difficulties and key analytical techniques in deriving \eqref{h117}-\eqref{hh118}, compared with previous works  in \cite{W 12, TZ 11, GTY 16}.
The main difficulty comes from the Allen-Cahn  equation $\eqref{h17}_{3,4}$, rewritten it as a second order nonlinear parabolic equation
\begin{equation}
\label{h120}
\phi_t-\f{\epsilon}{\rho^2}\Delta\phi-\f{1}{\epsilon\rho}\lf(1-\phi^2\rg)\phi=- \mathbf{u}\cdot\nabla\phi,\end{equation}where  the strong nonlinear term $-\lf(1-\phi^2\rg)\phi$ makes a trouble for desired estimates  because  $\|\phi(t)\|_{L^\infty(\mathbb{R}^3)}$
is not small and $\phi\notin L^p(\mathbb{R}^3)$ for any $1\leq p<\infty$ due to $\eqref{hh19}$. Moreover, the antiderivative of this nonlinear term, that is the phase-phase interfacial free energy density, is not a convex function, this makes it difficult to obtain the upper bound estimate for $L^2([0,T])$ norm of  $\|\nabla\phi(t)\|$. On the other hand, the coupling between the Navier-Stokes equations $\eqref{h17}_{1,2}$ and the Allen-Cahn equation  \eqref{h120} also bring trouble to get the density estimation. In order to overcome these difficulties, we first
find that $\int_0^T\|\nabla \phi(\tau)\|^2 d\tau$ can be controlled by the  small perturbations for the initial energy.
%, that is, we should assume  $\| (\rho_0-\bar{\rho},\mathbf{u}_0)\|_{H^{3}}+\| \nabla\phi_0\|_{H^{2}}+\|\phi^2_0-1\|$ small enough (see Lemma \ref{lem31}).
Further, we obtained the upper bound estimate of $\|\phi\|_{L^\infty(\mathbb{R}^3)}$ by using the method of energy estimates, (see Lemma \ref{lem31}).  Based on these facts, we complete the estimate \eqref{hr329} for $\nabla\phi$ at  $k-$th level, where a new $L^p-$estimate \eqref{hh20}  for nonlinear functions is used essentially together with Gagliardo-Nirenberg inequality \eqref{h20}  (Lemma \ref{lem34}). By the same lines as in \cite{W 12, TZ 11, GTY 16}, we derive the estimate \eqref{hr316} and \eqref{hr326}  for $(\rho-\bar{\rho}, \mathbf{u})$ at  $k-$th level (Lemma \ref{lem32} and Lemma \ref{lem33}). Combining Lemmas  \ref{lem31}-\ref{lem33}, we could obtain \eqref{h117} (see Subsection \ref{subs1}).
Next, multiplying  \eqref{h120} by $2\phi$ and applying $\Lambda^{-s}$  to the resulting equality \eqref{hr330}, we could obtain the following type of inequality
$$\f{d}{dt}\big\|\Lambda^{-s} \lf(\phi^2-1,\nabla\phi \rg)\big\|^2+\big\|\Lambda^{-s} \Delta\phi\big\|^2+\big\|\Lambda^{-s} \nabla\lf(\phi^2-1\rg)\big\|^2\lesssim \cdots \cdots, $$which is a key point in completing the negative Sobolev norm estimates \eqref{h118} and \eqref{hh118} (see \eqref{hh44} and Section \ref{sec2}).

%\section{Preliminaries}
%\setcounter{equation}{0}

\section{The local existence and a series of energy estimates}\label{sec2}
\setcounter{equation}{0}
This section is devoted to establish the local existence for the solutions of the the Cauchy problem \eqref{h17}-\eqref{hh19}, and to derive the a series of energy estimates which  play a key role in the process of extending the local solution to the global solution.
For any interval $I\subset[0,\infty)$, and $\forall m>0,M>0$,   we suppose  that
$(\sigma, \mathbf{u}, \phi)\in X_{m,M}(I)$ is the solution to the system
\eqref{h17}-\eqref{hh19}, where the solution space $X_{m,M}(I)$  is defined as follows
\begin{equation}\label{h310}
\begin{aligned}
X_{m,M}(I)=&\Big\{(\sigma,\mathbf{u}, \phi)\Big|\,(\sigma,\mathbf{u})\in C(\emph{I};H^3(\mathbb{R}^3)),\,\,\,\nabla\phi \in C(I;H^2(\mathbb{R}^3)),\\
& \phi^2-1\in C(I;L^{2}(\mathbb{R}^3)),\,\,\,\nabla\sigma \in L^2(I;H^2(\mathbb{R}^3)),%\nabla(\phi^2-1) \in L^2(I;L^{2}(\mathbb{R}^3)),
\\
&\nabla\phi \in L^2(I;H^3(\mathbb{R}^3)),\,\,\,\nabla \mathbf{u}\in L^2(I;H^3(\mathbb{R}^3))\\
&\sup_{ t \in I}\lf(\| (\sigma,\mathbf{u})\|_{H^{3}}+\| \nabla\phi\|_{H^{2}}+\|\phi^2(t)-1\|\rg)\leq M,\\
& \inf_{\mathbf{x}\in \mathbb{R}^3, t\in I}\phi^2(\mathbf{x},t)-\frac{1}{3}>m,\inf_{\mathbf{x}\in \mathbb{R}^3, t\in I}\rho(\mathbf{x},t)>m
\Big\}.
\end{aligned}\end{equation}
%In particular, let
%\begin{equation}\label{X}
% X_{0,\infty}=\bigcup_{m>0,M>0}X_{m,M}(I).
%\end{equation}
% Also, we assume a priori that for sufficiently small $M>0$,
	%\begin{equation}
%\label{h310}
%\sup_{0\leq t \leq
%	T}\lf(\| (\sigma,\mathbf{u})\|_{H^{3}}+\| \nabla\phi\|_{H^{2}}+\|\phi^2(t)-1\|\rg)\leq
%M.\end{equation}
We are now in a position to establish the existence and uniqueness for the local strong solutions$(\rho,\mathbf{u},\phi)$ of the Cauchy problem \eqref{h17}-\eqref{hh19}.
\begin{proposition}  \label{pro311}
	\textbf{(local existence). }Assume that \eqref{h14},\eqref{h110},\eqref{rho}-\eqref{phi}.   Let $\|(\rho_0-\bar{\rho}, \mathbf{u}_0)\|_{H^{3}} +\| \nabla\phi_0\|_{H^{2}}+\|\phi_0^2-1\|\leq M$, $\inf\limits_{\mathbf{x}\in \mathbb{R}^3}\phi_0^2(\mathbf{x})-\frac{1}{3}>m>0$ and $\inf\limits_{\mathbf{x}\in \mathbb{R}^3 }\rho_0(\mathbf{x})>m>0$, then there exists $T^*$ small enough, such that, the Cauchy problem \eqref{h17}-\eqref{hh19} admits a unique solution $(\rho,	\mathbf{u}, \phi)\in X_{\frac{m}{2},2M}\big([0,T^*]\big)$. satisfying
$$(\rho-\bar{\rho}, \mathbf{u}) \in C([0,T^*];H^3(\mathbb{R}^3)),\ \nabla\phi \in C([0,T^*];H^2(\mathbb{R}^3)), \phi^2-1\in C([0,T^*];L^{2}(\mathbb{R}^3)),$$
$$\nabla\rho\in L^2([0,T^*];H^2(\mathbb{R}^3)), \nabla\phi \in L^2([0,T^*];H^3(\mathbb{R}^3)),\ \nabla \mathbf{u}\in L^2([0,T^*];H^3(\mathbb{R}^3)).$$
%	Moreover,
%	\begin{equation}\label{h18}
%		\begin{aligned}
%	&\sup_{t\in[0,T^*]}\|(\rho-\bar{\rho}, \mathbf{u})(t)\|^2_{H^3}+ \|\nabla \phi(t)\|^2_{H^2}+ \|\phi^2(t)-1\|^2\\
%&+\int_0^{T^*}\|\nabla\rho\|^2_{H^2} d\tau+ \int_0^{T^*}\|(\nabla\mathbf{u},\nabla\phi\|^2_{H^3} d\tau\\
%	&\leq C\lf(\|\rho_0-\bar{\rho}\|^2_{H^3}+\|\mathbf{u}_0\|^2_{H^3}+ \|\nabla \phi_0\|^2_{H^2}\rg).
%	\end{aligned}\end{equation} where $C$ is a positive constant.
	\end{proposition}

Proposition \eqref{pro311}  can be obtained by the Schauder's fixed point method.  The proof is standard and we omit here.
Now, in order to get the global solution, on the basis of the existence and uniqueness of local solutions, we will give the a prior estimate by the following lemmas in this section.
Setting $\sigma=\rho-\bar{\rho},$
and using   $${\rm div}\lf(\nabla\phi\otimes \nabla\phi\rg)=\nabla\lf(\f{|\nabla\phi|^2}{2}\rg)+\nabla\phi\Delta\phi,$$  we
reformulate the system \eqref{h17} as
\begin{equation}\label{342}\left\{\begin{aligned}
&\sigma_t+\bar{\rho}{\rm div} \mathbf{u}=g_1,\\
&\mathbf{u}_t-\nu\bar{\rho}^{-1}\Delta\mathbf{u}-(\nu+\lambda)\bar{\rho}^{-1}\nabla{\rm
	div}
\mathbf{u}+\f{p'(\bar{\rho})}{\bar{\rho}}\nabla\sigma+\frac{\epsilon}{\bar\rho}\nabla\phi\Delta\phi=\mathbf{g}_2,\\
&\rho \phi_t+\rho \mathbf{u}\cdot\nabla\phi
=-\mu, \quad   \rho\mu=\f{\rho}{\epsilon}\lf(\phi^3-\phi\rg)-\epsilon\Delta\phi,\end{aligned}\right.
\end{equation}where $g_1$ and $\mathbf{g}_2$ are defined respectively by
\begin{equation}\label{343}\begin{aligned}
&g_1=-{\rm div}(\sigma \mathbf{u}),\\
&\mathbf{g}_2=-
(\mathbf{u}\cdot\nabla)\mathbf{u}+h_1(\sigma)\nabla\sigma
-h_2(\sigma)\Big(\nu\Delta\mathbf{u}+(\nu+\lambda)\nabla{\rm
	div} \mathbf{u}-\epsilon\nabla\phi\Delta\phi\Big),\end{aligned}
\end{equation}
here $h_1(\sigma)=\f{p'(\bar{\rho})}{\bar{\rho}}-\f{p'(\rho)}
{\rho }$ and $h_2(\sigma)=\bar{\rho}^{-1}-\rho^{-1}$.
Considering the definition \eqref{h310}, combining with Sobolev embedding theorem $H^2\hookrightarrow C^0$, we can choose $M_0>0$, such that, $\forall0<M<M_0$,
\begin{equation}\label{h315}
\frac{\bar\rho}{2}\leq\rho(\mathbf{x},t)\leq 2\bar{\rho}, \qquad 3\phi^2-1>m_0\overset{\text{def}}{=}\inf_{\mathbf{x}\in\mathbb{R}^3}(3\phi_0^2-1).
\end{equation}
% We first give the following  estimate which plays an essential role:

The following lemma 3.1 is the basic energy inequality and the  gradient estimation about the phase field $\phi$.
\begin{lemma}\label{lem31}  Under the  assumption \eqref{h310}, it holds   that
	\begin{equation}
	\label{hr32}
	\begin{aligned}
	\|(\sigma, \mathbf{u}, \phi^2-1, \nabla\phi)(t)\|^2+\int_0^t \|(\mu, \nabla \mathbf{u},\Delta\phi,\nabla\phi ) \|^2d\tau\leq  C_0 \|(\sigma, \mathbf{u}, \phi^2-1, \nabla\phi)(0)\|^2.
	\end{aligned}\end{equation}
%\begin{equation}\label{hr32-1}
%\begin{aligned}
%\f{d}{dt}&\int\lf(\f{1}{2}\rho	\mathbf{u}^2+G(\rho)+\frac{\epsilon}{2}|\nabla
%	\phi|^2+\frac{\rho}{4\epsilon}(\phi^2-1)^2\rg)dx +\nu\|\nabla\mathbf{u}\|^2+\|\mu \|^2\\
%&+\f{\epsilon^2}{16\bar{\rho}^2}\|\Delta\phi\|^2+\f{1}{64\bar{\rho}}\|\nabla\lf(\phi^2-1\rg)\|^2\leq M\|\nabla\phi\|^2,\end{aligned}
%\end{equation}
and
	\begin{equation}
	\label{hr33}
	\|\phi(t)\|_{L^\infty}\lesssim 1,	\end{equation} where  \begin{equation}
	\label{energy}
	G(\rho)=\rho\int_{\bar{\rho}}^\rho\f{p(z)-p(\bar{\rho})}{z^2}dz, \quad\rho>0.
	\end{equation}
	\end{lemma}
\begin{proof} Noticing that $$\rho G'(\rho)=G(\rho)+(p(\rho)-p(\bar{\rho})),\quad \rho G''(\rho)=p'(\rho),$$
$$G(\rho)_{t}+{\rm div}(G(\rho) \mathbf{u})+(p(\rho)-p(\bar{\rho})){\rm div}\mathbf{u}=
0,$$ and using
$\eqref{h17}_2$, we have
\begin{equation}\label{3.9-3}\begin{aligned}
		&\f{d}{dt}\int\lf(\f{1}{2}\rho
		\mathbf{u}^2+G(\rho)\rg)d\mathbf{x}+\nu\int|\nabla \mathbf{u}|^2d\mathbf{x}\\
		&+(\nu+\lambda)\int |{\rm div}\mathbf{u}|^2d\mathbf{x}-\epsilon\int\mathbf{u}\cdot\nabla\phi\Delta\phi d\mathbf{x}=0.\end{aligned}
\end{equation}
Multiplying $\eqref{342}_3$  by $\mu$ and using $\eqref{342}_4$ yields that
\begin{equation}\label{h37}\f{\epsilon}{2}\f{d}{dt}\int|\nabla\phi|^2d\mathbf{x}+\f{1}{4\epsilon}\f{d}{dt}\int\rho\lf(\phi^2-1\rg)^2d\mathbf{x} +\int|\mu |^2d\mathbf{x}=-\epsilon\int \mathbf{u}\cdot \nabla\phi \Delta\phi d\mathbf{x},
\end{equation}where we used
$$\f{1}{4}\int\lf(\phi^2-1\rg)^2{\rm div}(\rho \mathbf{u})d\mathbf{x}=-\int\rho \mathbf{u}\cdot \nabla\phi \lf(\phi^3-\phi\rg)d\mathbf{x}.$$
Adding \eqref{3.9-3} and \eqref{h37}, and using \eqref{h14},
we get
\begin{equation}\label{h38}\f{d}{dt}\int\lf(\f{1}{2}\rho
\mathbf{u}^2+G(\rho)+\frac{\epsilon}{2}|\nabla
\phi|^2+\frac{\rho}{4\epsilon}(\phi^2-1)^2\rg)d\mathbf{x} +\nu\|\nabla\mathbf{u}\|^2+\|
\mu \|^2\leq 0.
\end{equation}
Using  \eqref{energy}, \eqref{h310} and \eqref{h315}, we have
\begin{equation}\label{h330}
c_{\bar{\rho}}(\rho-\bar{\rho})^2\leq G(\rho)\leq C_{\bar{\rho}}(\rho-\bar{\rho})^2.\end{equation}
Therefore, by \eqref{h38},
\eqref{h330} and \eqref{h315}, we obtain from \eqref{h38} that
\begin{equation}
\label{hh32}\begin{aligned}
&\|(\sigma, \mathbf{u}, \phi^2-1, \nabla\phi)(t)\|^2+\int_0^t \|(\mu, \nabla \mathbf{u})  \|^2d\tau\leq  C_0 \|(\sigma, \mathbf{u}, \phi^2-1, \nabla\phi)(0)\|^2.
\end{aligned}
\end{equation}
We rewrite $\eqref{342}_{3,4}$ as
\begin{equation}\label{phi111} %\rho\phi_t+ \rho\mathbf{u}\cdot\nabla\phi=\frac{\epsilon}{\rho}\Delta\phi-\frac{1}{\epsilon}(\phi^3-\phi).
\rho^2\phi_t+ \rho^2\mathbf{u}\cdot\nabla\phi=\epsilon\Delta\phi-\frac{\rho}{\epsilon}(\phi^3-\phi).
\end{equation}
Multiplying it by $2\phi$, we have
\begin{equation}\label{phi222} %\rho(\phi^2-1)_t+ \rho\mathbf{u}\cdot\nabla(\phi^2-1)-\frac{\epsilon}{\rho}\Delta(\phi^2-1)+\frac{2}{\epsilon}(\phi^2-1)=-\frac{2\epsilon}{\rho}|\nabla\phi|^2-\frac{2}{\epsilon}(\phi^2-1)^2\leq 0,
\rho^2(\phi^2-1)_t+ \rho^2\mathbf{u}\cdot\nabla(\phi^2-1)-\epsilon\Delta(\phi^2-1)+\frac{2\rho}{\epsilon}(\phi^2-1)=-2\epsilon|\nabla\phi|^2-\frac{2\rho}{\epsilon}(\phi^2-1)^2\leq 0.
\end{equation}
By using the maximum principle for parabolic equation (see Lemma 2.1 in \cite{p2005}) and \eqref{h315}, we obtain
\begin{equation}\label{phi333}
\displaystyle \phi^2-1\leq 0,
\end{equation} which yields \eqref{hr33}.
% Setting $\Omega_n=[n, n+1)^3$ for  $\forall n=0, \pm1, \pm 2,\cdots,$
% we obtain from \eqref{hh32}
% $$\begin{aligned}
% \int_{\Omega_n}\phi^4 d\mathbf{x}&\leq 2\int_{\Omega_n}\phi^2 d\mathbf{x}+E_0\leq\f{1}{2}\int_{\Omega_n}\phi^4 d\mathbf{x}+(2+E_0),
% \end{aligned}$$which implies
% $$\int_{\Omega_n}\phi^4 d\mathbf{x}\leq 2(2+E_0),$$
% where $E_0= C_0 \|(\sigma, \mathbf{u}, \phi^2-1, \nabla\phi)(0)\|^2.$ Therefore, we have
% \begin{equation}\label{h317}
% \int_{\Omega_n}\phi(\mathbf{x},t)d\mathbf{x}\leq \lf(\int_{\Omega_n}\phi^4 d\mathbf{x}\rg)^{\f{1}{4}}\leq \lf(4+2E_0\rg)^{\f{1}{4}}.\end{equation}
% Using \eqref{h317} yields that for any $(\mathbf{x},t)\in \Omega_n\times [0,T]$, \begin{equation}\label{h318}
% \begin{aligned}
% |\phi(\mathbf{x},t)|&\leq {\Big|}\int_{\Omega_n}(\phi(\mathbf{x},t)-\phi(\mathbf{y},t)) d\mathbf{y}{\Big|}+{\Big|}\int_{\Omega_n}\phi(\mathbf{y},t) d\mathbf{y}{\Big|}\\
% &\leq C\lf(\int_{\Omega_n}|\nabla\phi|^2\rg)^\f{1}{2}+\lf(4+2E_0\rg)^{\f{1}{4}},
% \end{aligned}\end{equation}where $C$ is independent of $n$. By \eqref{h318} and \eqref{h38}, we get \eqref{hr33}.
On the other hand, multiplying $\eqref{342}_{4}$ by $-\Delta\phi$, we have
\begin{equation}\label{hg320}
\begin{aligned}
&\epsilon\|\Delta\phi\|^2+\f{1}{\epsilon}\underline{\int \rho(3\phi^2-1)|\nabla\lf(\phi\rg)|^2d\mathbf{x}}_{I_1}
\\&%-\f{1}{\epsilon}\underline{\int\rho \lf(\phi^2-1\rg) |\nabla\phi|^2 dx}_{I_2}
=-\underline{\int\rho \mu \Delta \phi d\mathbf{x}}_{I_2}-\f{1}{\epsilon}\underline{\int  \lf(\phi^2-1\rg) \phi\nabla\phi\cdot\nabla\sigma d\mathbf{x}}_{I_3}.
\end{aligned}
\end{equation}
For the estimates  on $I_{i}(i=1,2,3)$, we have
$$\begin{aligned}
&I_1\geq\frac{\bar\rho c_0}{4\epsilon}\int\|\nabla\phi\|^2d\mathbf{x},\qquad I_2\overset{\eqref{h315}}\leq2\bar{\rho} \|\mu\| \|\Delta\phi\| \leq \f{\epsilon}{4}\|\Delta\phi\|^2+\f{4\bar{\rho}^2}{\epsilon}\|\mu\| ^2,\\
%&\begin{aligned}
%I_2&\lesssim \|\lf(\phi^2-1\rg)\|_{L^6}\|\nabla\phi\|_{L^3} \|\nabla\phi\|\overset{\eqref{h20}}\lesssim \|\nabla\lf(\phi^2-1\rg)\| \|\nabla\phi\|^\f{3}{2}\|\nabla^2\phi\|^\f{1}{2} \\
%&\overset{\eqref{h310}}\lesssim M\lf(\|\nabla^2\phi\|^2+\|\nabla\phi\|^2+\|\nabla\lf(\phi^2-1\rg)\|^2\rg),\end{aligned}\\
&
I_3 \lesssim\|\lf(\phi^2-1\rg)\|_{L^6}\|\phi\nabla\phi\|\|\nabla\sigma\|_{L^3}\overset{\eqref{h310}}\lesssim M\|\nabla\phi\|^2.
\end{aligned}
$$
Substituting the estimates  on $I_{i}(i=1,2,3)$ into \eqref{hg320}, and using \eqref{h315}, for $M$ suitable small, we get
\begin{equation}\label{hg321}
\epsilon^2\|\Delta\phi\|^2+\|\nabla\phi\|^2\leq \|\mu\| ^2.
\end{equation}
By \eqref{h38} and \eqref{hg321}, we get \eqref{hr32}.
The proof of Lemma \ref{lem31} is completed.\end{proof}

\vspace{0.2cm} The following lemma is a higher-order estimate of the gradient of the phase field $\phi$.
\begin{lemma}\label{lem34}  Under the  the  assumption \eqref{h310}, it holds   that
	\begin{equation}
	\label{hr329}
	\f{d}{dt}\|\nabla^{k+1} \phi\|^2+\|\nabla^{k+2}\phi\|^2	\lesssim M\lf(\|\nabla^{k+1}	\sigma\|^2+ \|\nabla^{k+1}\phi\|^2 +	\|\nabla^{k+2} \mathbf{u}\|^2\rg),
	\end{equation}for $k=1, 2$.
\end{lemma}
\begin{proof} We rewrite $\eqref{342}_{3,4}$ as
\begin{equation}\label{hg330} \phi_t+ \mathbf{u}\cdot\nabla\phi-\f{\epsilon}{\rho^2}\Delta\phi+ \f{\phi^2-1}{\epsilon\rho}\phi=0.\end{equation}
%Integrating \eqref{hg330} multiplied  by $\phi$ over $\mathbb{R}^3$, combining with Gagliardo-Nirenberg inequality, similar as the derivation of \eqref{hr32-1}, we obtain
%\begin{equation}\label{hg328-1}\begin{aligned}
%\f{1}{2}&\f{d}{dt}\|\phi^2-1\|_{L^1}+\int\f{\epsilon}{\rho^2}|\nabla\phi|^2dx\\
%&= \int (\mathbf{u}\cdot\nabla\phi) \phi dx+\epsilon\int \f{2}{\rho^3}\nabla\rho \nabla\phi \phi dx+\f{1}{\epsilon}\int\f{1}{\rho}(\phi^2-1) \phi^2 dx\\
%&\leq\|\nabla(\phi^2-1)\|_{L^6}\|\phi\|_{L^3}\|\phi\|_{L^2}+
%\end{aligned}
%\end{equation}
Applying $\nabla^k$ to \eqref{hg330}  and
multiplying it by $-\Delta \nabla^k \phi$, we have
\begin{equation}\label{hg328}\begin{aligned}
\f{1}{2}&\f{d}{dt}\|\nabla^{k+1}\phi\|^2+\int\f{\epsilon}{\rho^2}|\nabla^{k}\Delta \phi|^2d\mathbf{x}\\
&= \underline{\int \nabla^k (\mathbf{u}\cdot\nabla\phi) \Delta \nabla^k\phi d\mathbf{x}}_{I_{4}}-\epsilon\underline{\sum_{1\leq l\leq k}
	C_{k}^{l}\int \nabla^{l}\lf(\f{1}{\rho^2}\rg) \nabla^{k-l}\Delta\phi \nabla^{k}\Delta\phi d\mathbf{x}}_{I_{5}}\\
&+\f{1}{\epsilon}\underline{\sum_{0\leq l\leq k}
	C_{k}^{l}\int \nabla^{l}\lf(\f{\phi^2-1}{\rho}\rg) \nabla^{k-l} \phi \nabla^{k}\Delta\phi d\mathbf{x}}_{I_{6}}.
\end{aligned}
\end{equation}
We estimate  $I_{i} (i=4,5, 6)$. For  $I_{4}$, we have
\begin{eqnarray*}
I_{4}&\displaystyle=\sum_{0\leq l\leq k}C^l_{k}\int \nabla^{l}\mathbf{u}\cdot\nabla^{k-l+1}\phi \Delta \nabla^k\phi d\mathbf{x} \\
&\displaystyle\lesssim \sum_{0\leq l\leq k} \|\nabla^{l}\mathbf{u}\cdot\nabla^{k-l+1}\phi\|\|\nabla^{k+2}\phi\|.
\end{eqnarray*}
If $l\leq \lf[\f{k+1}{2}\rg]$,   we get
$$
\begin{aligned}
&\displaystyle\|\nabla^{l}\mathbf{u}\cdot\nabla^{k-l+1}\phi\|\lesssim \|\nabla^{l}\mathbf{u}\|_{L^3}\|\nabla^{k-l+1}\phi\|_{L^6} \\
&\displaystyle\overset{\eqref{h20}}\lesssim \|\nabla^\alpha\mathbf{u}\|^{1-\f{l}{k+1}}\|\nabla^{k+2}\mathbf{u}\|^{\f{l}{k+1}}\|\nabla\phi\|^{\f{l}{k+1}}\|\nabla^{k+2}\phi\|^{1-\f{l}{k+1}}
\\
&\displaystyle\overset{\eqref{h310}}\lesssim M\|\nabla^{k+2}\mathbf{u}\|^{\f{l}{k+1}}\|\nabla^{k+2}\phi\|^{1-\f{l}{k+1}}\lesssim M\lf(\|\nabla^{k+2}\mathbf{u}\|+\|\nabla^{k+2}\phi\|\rg),
\end{aligned}$$
where $\alpha$ is defined by
\begin{eqnarray*}
&&\displaystyle\frac{l-1}{3}=\lf(\frac{\alpha}{3}-\frac{1}{2}\rg)\lf(1-\f{l}{k+1}\rg)+\lf(\frac{k+2}{3}-\frac{1}{2}\rg)\f{l}{k+1} \\
&&\Longrightarrow \alpha=\f{3}{2}+\f{l}{2(k+1-l)}\in \lf[\f{3}{2}, \f{5}{2}\rg].
\end{eqnarray*}
If $\lf[\f{k+1}{2}\rg]+1\leq l\leq k$ (if $k\leq \lf[\f{k+1}{2}\rg]+1$, then it's nothing in this case, and hereafter, etc.), we get
$$
\begin{aligned}
&\displaystyle\|\nabla^{l}\mathbf{u}\cdot\nabla^{k-l+1}\phi\|\lesssim \|\nabla^{l}\mathbf{u}\|_{L^6}\|\nabla^{k-l+1}\phi\|_{L^3} \\
&\displaystyle\overset{\eqref{h20}}\lesssim \|\mathbf{u}\|^{1-\f{l+1}{k+2}}\|\nabla^{k+2}\mathbf{u}\|^{\f{l+1}{k+2}}\|\nabla^\alpha\phi\|^{\f{l+1}{k+2}}\|\nabla^{k+2}\phi\|^{1-\f{l+1}{k+2}}
\\
&\displaystyle\overset{\eqref{h310}}\lesssim M\|\nabla^{k+2}\mathbf{u}\|^{\f{l+1}{k+2}}\|\nabla^{k+2}\phi\|^{1-\f{l+1}{k+2}}\lesssim M\lf(\|\nabla^{k+2}\mathbf{u}\|+\|\nabla^{k+2}\phi\|\rg),
\end{aligned}$$
where $\alpha$ is defined by
\begin{eqnarray*}
&&\displaystyle   \frac{k-l}{3}=\lf(\frac{\alpha}{3}-\frac{1}{2}\rg)\f{l+1}{k+2}+\lf(\frac{k+2}{3}-\frac{1}{2}\rg)\lf(1-\f{l+1}{k+2}\rg) \\
&&\displaystyle \Longrightarrow \alpha=\f{3(k+2)}{2(l+1)}\in \lf[\f{3}{2},3\rg].
  \end{eqnarray*}
Therefore, we obtain
\begin{equation}\label{hr323}
|I_{4}|\lesssim  M\lf(\|\nabla^{k+2}\mathbf{u}\|^2+\|\nabla^{k+2}\phi\|^2\rg).\end{equation}
Also, for  $I_{5}$, we have
\begin{eqnarray*}
I_{5}&\displaystyle=\sum_{1\leq l\leq k}C_{k}^{l}\int \nabla^{l}\lf(\f{1}{\rho^2}\rg) \nabla^{k-l}\Delta\phi \nabla^{k}\Delta\phi d\mathbf{x}\\
 &\displaystyle\lesssim \sum_{1\leq l\leq k} \| \nabla^{l}\lf(\f{1}{\rho^2}\rg) \nabla^{k-l}\Delta\phi\|\|\nabla^{k+2}\phi\|.
\end{eqnarray*}
If $1\leq l\leq \lf[\f{k+1}{2}\rg]$,   we get
$$
\begin{aligned}
&\displaystyle\|\nabla^{l}\lf(\f{1}{\rho^2}\rg) \nabla^{k-l}\Delta\phi\|\overset{\eqref{hh20}}\lesssim \|\nabla^{l}\sigma\|_{L^3}\|\nabla^{k-l+2}\phi\|_{L^6} \\
&\displaystyle\overset{\eqref{h20}}\lesssim \|\nabla^\alpha\sigma\|^{1-\f{l-1}{k+1}}\|\nabla^{k+1}\sigma\|^{\f{l-1}{k+1}}\|\nabla\phi\|^{\f{l-1}{k+1}}\|\nabla^{k+2}\phi\|^{1-\f{l-1}{k+1}}
\\
&\displaystyle\overset{\eqref{h310}}\lesssim M\|\nabla^{k+1}\sigma\|^{\f{l-1}{k+1}}\|\nabla^{k+2}\phi\|^{1-\f{l-1}{k+1}}\lesssim M\lf(\|\nabla^{k+1}\sigma\|+\|\nabla^{k+2}\phi\|\rg),
\end{aligned}$$
where $\alpha$ is defined by
\begin{eqnarray*}
 &&\displaystyle \frac{l-1}{3}=\lf(\frac{\alpha}{3}-\frac{1}{2}\rg)\lf(1-\f{l-1}{k+1}\rg)+\lf(\frac{k+1}{3}-\frac{1}{2}\rg)\f{l-1}{k+1} \\
 &&\displaystyle\Longrightarrow \alpha=\f{3}{2}+\f{l-1}{2(k+2-l)}\in \lf[\f{3}{2}, \f{5}{2}\rg].
\end{eqnarray*}
If $\lf[\f{k+1}{2}\rg]+1\leq l\leq k$, we get
$$
\begin{aligned}
&\|\nabla^{l}\lf(\f{1}{\rho^2}\rg) \nabla^{k-l}\Delta\phi\|\overset{\eqref{hh20}}\lesssim \|\nabla^{l}\sigma\|_{L^6}\|\nabla^{k-l+2}\phi\|_{L^3} \\
&\overset{\eqref{h20}}\lesssim \|\sigma\|^{1-\f{l+1}{k+1}}\|\nabla^{k+1}\sigma\|^{\f{l+1}{k+1}}\|\nabla^\alpha\phi\|^{\f{l+1}{k+1}}\|\nabla^{k+2}\phi\|^{1-\f{l+1}{k+1}}
\\
&\overset{\eqref{h310}}\lesssim M\|\nabla^{k+1}\sigma\|^{\f{l+1}{k+1}}\|\nabla^{k+2}\phi\|^{1-\f{l+1}{k+1}}\lesssim M\lf(\|\nabla^{k+1}\sigma\|+\|\nabla^{k+2}\phi\|\rg),
\end{aligned}$$
where $\alpha$ is defined by
\begin{eqnarray*}
&&\displaystyle\frac{k-l+1}{3}=\lf(\frac{\alpha}{3}-\frac{1}{2}\rg)\f{l+1}{k+1}+\lf(\frac{k+2}{3}-\frac{1}{2}\rg)\lf(1-\f{l+1}{k+1}\rg) \\
&&\displaystyle\Longrightarrow \alpha=1+\f{k+1}{2(l+1)}\in \lf[\f{3}{2},3\rg].
\end{eqnarray*}
Therefore,  we obtain
\begin{equation}\label{hr333}
|I_{5}|\lesssim  M\lf(\|\nabla^{k+1}\sigma\|^2+\|\nabla^{k+2}\phi\|^2\rg).\end{equation}
To estimate $I_{6}$, we rewrite it as
\begin{equation}\label{hh326}\begin{aligned}
I_{6}&=-\underline{\int\lf(\f{\phi^2-1}{\rho}\rg)|\nabla^{k+1}\phi|^2 d\mathbf{x}}_{I_{6}^1}-\underline{\int\nabla\lf(\f{\phi^2-1}{\rho}\rg)\nabla^{k}\phi\cdot\nabla^{k+1}\phi d\mathbf{x}}_{I_{6}^2}\\
&\qquad+\f{1}{\epsilon}\underline{\sum_{1\leq l\leq k}
	C_{k}^{l}\int \nabla^{l}\lf(\f{\phi^2-1}{\rho}\rg) \nabla^{k-l} \phi \nabla^{k}\Delta\phi d\mathbf{x}}_{I_{6}^3}.\end{aligned}
\end{equation}
For  $I_{6}^1$, we have
$$\begin{aligned}
I_{6}^1&\displaystyle\lesssim  \|\phi^2-1\|_{L^3}\|\nabla^{k+1}\phi\|\|\nabla^{k+1}\phi\|_{L^6}\\
&\displaystyle\overset{\eqref{h20}}\lesssim  \|\phi^2-1\|^{\f{1}{2}}\|\nabla\lf(\phi^2-1\rg)\|^\f{1}{2}\|\nabla^{k+1}\phi\|\|\nabla^{k+2}\phi\|\\
&\displaystyle\overset{\eqref{hr33}, \eqref{h310} }\lesssim M\lf(\|\nabla^{k+1}\phi\|^2+\|\nabla^{k+2}\phi\|^2\rg).
\end{aligned}
$$
For  $I_{6}^2$, we have
$$\begin{aligned}
I_{6}^2&\lesssim {\Big\|}\nabla\lf(\f{\phi^2-1}{\rho}\rg){\Big\|}  \|\nabla^{k}\phi\|_{L^6}\|\nabla^{k+1}\phi\|_{L^3}\\
&\overset{\eqref{hh20}, \eqref{h20}}\lesssim  \lf(\|\nabla\phi\|+\|\nabla\sigma\|\rg) \|\nabla^{k+1}\phi\|^\f{3}{2}\|\nabla^{k+2}\phi\|^\f{1}{2}\\
&\overset{\eqref{h310}}\lesssim M\lf(\|\nabla^{k+1}\phi\|^2 +\|\nabla^{k+2}\phi\|^2 \rg).
\end{aligned}
$$
For  $I_{6}^3$, we have
$$
I_{6}^3\lesssim \sum_{1\leq l\leq k} \|\nabla^{l}\lf(\f{\phi^2-1}{\rho}\rg) \nabla^{k-l} \phi\|\|\nabla^{k+2}\phi\|.
$$
If $1\leq l\leq \lf[\f{k}{2}\rg]$,   we get
$$
\begin{aligned}
&\|\nabla^{l}\lf(\f{\phi^2-1}{\rho}\rg) \nabla^{k-l} \phi\|\overset{\eqref{hh20}, \eqref{hr33}}\lesssim \lf(\|\nabla^{l}\sigma\|_{L^3}+\|\nabla^{l}\phi\|_{L^3}\rg)\|\nabla^{k-l}\phi\|_{L^6} \\
&\overset{\eqref{h20}}\lesssim\lf( \|\nabla^\alpha\sigma\|^{1-\f{l}{k}}\|\nabla^{k}\sigma\|^{\f{l}{k}}+ \|\nabla^\alpha\phi\|^{1-\f{l}{k}}\|\nabla^{k}\phi\|^{\f{l}{k}}\rg)\|\nabla\phi\|^{\f{l}{k}}\|\nabla^{k+1}\phi\|^{1-\f{l}{k}}
\\
&\overset{\eqref{h310}}\lesssim M\lf(\|\nabla^{k+1}\sigma\|+\|\nabla^{k+1}\phi\|\rg),
\end{aligned}$$
where $\alpha$ is defined by
$$
\frac{l-1}{3}=\lf(\frac{\alpha}{3}-\frac{1}{2}\rg)\lf(1-\f{l}{k}\rg)+\lf(\frac{k+1}{3}-\frac{1}{2}\rg)\f{l}{k}\Longrightarrow \alpha=\f{1}{2}+\f{l}{2k}\in \lf[\f{1}{2}, 1\rg].$$
If $\lf[\f{k}{2}\rg]+1\leq l\leq k$, we get
$$
\begin{aligned}
&\|\nabla^{l}\lf(\f{\phi^2-1}{\rho}\rg) \nabla^{k-l} \phi\|\overset{\eqref{hh20}, \eqref{hr33}}\lesssim \lf(\|\nabla^{l}\sigma\|_{L^6}+\|\nabla^{l}\phi\|_{L^6}\rg)\|\nabla^{k-l}\phi\|_{L^3} \\
&\overset{\eqref{h20}}\lesssim\lf(\|\sigma\|^{1-\f{l+1}{k+1}}\|\nabla^{k+1}\sigma\|^{\f{l+1}{k+1}}+\|\phi\|^{1-\f{l+1}{k+1}}\|\nabla^{k+1}\phi\|^{\f{l+1}{k+1}}\rg)\|\nabla^\alpha\phi\|^{\f{l+1}{k+1}}\|\nabla^{k+1}\phi\|^{1-\f{l+1}{k+1}}
\\
&\overset{\eqref{h310}}\lesssim M\lf(\|\nabla^{k+1}\sigma\|+\|\nabla^{k+1}\phi\|\rg),
\end{aligned}$$
where $\alpha$ is defined by
\begin{eqnarray*}
  &&\displaystyle\frac{k-l-1}{3}=\lf(\frac{\alpha}{3}-\frac{1}{2}\rg)\f{l+1}{k+1}+\lf(\frac{k+1}{3}-\frac{1}{2}\rg)\lf(1-\f{l+1}{k+1}\rg) \\
&&\displaystyle\Longrightarrow \alpha=\f{k+1}{2(l+1)}\in \lf[\f{1}{2},2\rg].
\end{eqnarray*}
Therefore,  we obtain
$$I_{6}^3\lesssim M\lf(\|\nabla^{k+1}\phi\|^2 +\|\nabla^{k+2}\phi\|^2 +\|\nabla^{k+1}\sigma\|^2 \rg).$$
Therefore,  we obtain from \eqref{hh326} and the estimates on $I_{9}^i (i=1,2,3)$ that
\begin{equation}\label{hh330}
|I_{6}|\lesssim  M\lf(\|\nabla^{k+1}\sigma\|^2+\|\nabla^{k+1}\phi\|^2+\|\nabla^{k+2}\phi\|^2\rg).\end{equation}
Substituting \eqref{hr323}, \eqref{hr333} and \eqref{hh330} into \eqref{hg328}, we have \eqref{hr329}.The proof of Lemma \ref{lem34} is completed.\end{proof}

The following lemma 3.3 is a higher-order estimate of  $(\sigma, \mathbf{u})$.
\begin{lemma}\label{lem32}  Under the  the  assumption \eqref{h310}, it holds   that,  for $k=0,1,2,3$,
	\begin{equation}
	\label{hr316}
	\f{d}{dt}\lf(\|\nabla^k\mathbf{u}\|^2
	+\|\nabla^k\sigma\|^2\rg)+\|\nabla^{k+1} \mathbf{u}\|^2\lesssim M\lf(\|\nabla^{k}
	\sigma\|^2+\|\nabla^{k+1}\phi\|^2\rg).
	\end{equation}
\end{lemma}
\begin{proof} Applying $\nabla^k $ to $\eqref{342}_1$ and $\eqref{342}_2$ yields respectively that
\begin{equation}\label{347}
\begin{aligned}
& \nabla^k \sigma_{t}+\bar{\rho}{\rm div}\nabla^k \mathbf{u}=\nabla^k  g_1, \\
&\nabla^k  \mathbf{u}_t-\frac{\nu}{\bar{\rho}}\,\Delta \nabla^k  \mathbf{u}
-\frac{(\nu+\lambda)}{\bar{\rho}}\nabla^{k+1}  {\rm div}\mathbf{u}+\f{p'(\bar{\rho})}
{\bar{\rho}}\nabla^{k+1}\sigma+\frac{\epsilon}{\bar\rho} \nabla^k(\nabla\phi \Delta \phi)
=\nabla^k\mathbf{g}_2.
\end{aligned}
\end{equation}
Multiplying $\eqref{347}_2$  by $\nabla^k
\mathbf{u}$,  and using $\eqref{347}_1$ and \eqref{343},  we have
\begin{equation}\label{3412}\begin{aligned}
\f{1}{2}\f{d}{dt}\lf(\|\nabla^k\mathbf{u}\|^2
+\f{p'(\bar{\rho})}{\bar{\rho}^2}\|\nabla^k\sigma\|^2\rg)+\int\lf(\frac{\nu }{\bar{\rho}}| \nabla^{k+1}\mathbf{u}|^2
+\frac{(\nu+\lambda)}{\bar{\rho}}|{\rm div}\nabla^k\mathbf{u}|^2\rg)d\mathbf{x}=I_{7},
\end{aligned}
\end{equation}where
\begin{equation}\label{hh333}\begin{aligned}
&I_{7}=\f{p'(\bar{\rho})}{\bar{\rho}^2}\underline{\int \nabla^{k}\sigma \nabla^{k}{\rm div}(\sigma\mathbf{u})d\mathbf{x}}_{I_{7}^1}-\underline{\int \nabla^k[(\mathbf{u},\nabla)\mathbf{u}-h_1(\sigma)\nabla\sigma]\cdot \nabla^k\mathbf{u}d\mathbf{x}}_{I_{7}^2}\\
&\qquad-\underline{\int \nabla^k[h_2(\sigma)\lf(\nu\Delta\mathbf{u}+(\nu+\lambda)\nabla{\rm
		div} \mathbf{u}-\epsilon\nabla\phi\Delta\phi\rg)]\cdot \nabla^k\mathbf{u}d\mathbf{x}}_{I_{7}^3}\\
&\qquad-\frac{\epsilon}{\bar\rho}\underline{\int
	\nabla^k(\nabla\phi \Delta \phi)\cdot \nabla^k\mathbf{u}d\mathbf{x}}_{I_{7}^4}
+\underline{\epsilon\int \nabla^k \big[h_2(\sigma)\nabla\phi\Delta\phi\big]\cdot \nabla^k\mathbf{u}d\mathbf{x}}_{I_{7}^5},\end{aligned}\end{equation}
\vspace{0.2cm} We will give the energy estimate of  $I_{7}$ below.
By the same lines as in \cite[Lemma 2.1]{W 12}, we can derive that
$$
I_{7}^i\lesssim  M\lf(\|\nabla^{k}
\sigma\|^2+
\|\nabla^{k+1} \mathbf{u}\|^2\rg)\quad \text{for}\,\,\, i=1, 2, 3.$$
For $I_{7}^4$,   it is easy to check that
\begin{eqnarray*}
     &\displaystyle I_{7}^4\lesssim  \|\nabla\phi\|\|\nabla^2\phi\|_{L^\infty}\|\nabla\mathbf{u}\|\\
   &\displaystyle\qquad\qquad\qquad\qquad\overset{\eqref{h310}}\lesssim  M\lf(\|\nabla\phi\|^2+\|\nabla \mathbf{u}\|^2\rg),\qquad \mathrm{for} \ k=0.\end{eqnarray*}
 and
 \begin{eqnarray*}
  &\displaystyle I_{7}^4\lesssim  \|\nabla\phi\|_{L^\infty}\|\nabla^{2}\phi\|\|\nabla^{2}\mathbf{u}\| \\
  &\displaystyle\qquad\qquad\qquad \overset{\eqref{h310}}\lesssim  M\lf(\|\nabla^{2} \phi\|^2+\|\nabla^{2} \mathbf{u}\|^2\rg),\qquad \mathrm{for} \ k=1.
 \end{eqnarray*}
When $k\geq 2$,
using  Leibniz formula yields that
\begin{eqnarray*}
  &\displaystyle I_{7}^4=-\int \nabla^{k-1}[\nabla\phi \Delta \phi]\cdot \nabla^{k+1}\mathbf{u}d\mathbf{x} \\
  &\displaystyle\qquad\qquad\lesssim \sum_{0\leq l\leq k-1} \|\nabla^{l+1}\phi\nabla^{k-l+1}\phi\|\|\nabla^{k+1}\mathbf{u}\|.
\end{eqnarray*}
If $l\leq \lf[\f{k+1}{2}\rg]$,   we get
$$
\begin{aligned}
&\|\nabla^{l+1}\phi\nabla^{k-l+1}\phi\|\lesssim \|\nabla^{l+1}\phi\|_{L^3}\|\nabla^{k-l+1}\phi\|_{L^6}  \\
&\overset{\eqref{h20}}\lesssim \|\nabla^\alpha\phi\|^{1-\f{l-1}{k-1}}\|\nabla^{k+1}\phi\|^{\f{l-1}{k-1}}\|\nabla^2\phi\|^{\f{l-1}{k-1}}\|\nabla^{k+1}\phi\|^{1-\f{l-1}{k-1}}\\
&\overset{\eqref{h310}}\lesssim M\|\nabla^{k+1}\phi\|,
\end{aligned}$$
where $\alpha$ is defined by
$$
\frac{l}{3}=\lf(\frac{\alpha}{3}-\frac{1}{2}\rg)\lf(1-\f{l-1}{k-1}\rg)+\lf(\frac{k+1}{3}-\frac{1}{2}\rg)\f{l-1}{k-1}\rightarrow \alpha=2-\f{k-1}{2(k-l)}\in \lf[1, 2\rg].$$
If $\lf[\f{k+1}{2}\rg]+1\leq l\leq k-1$,   we get
$$
\begin{aligned}
&\displaystyle\|\nabla^{l+1}\phi\nabla^{k-l+1}\phi\|\lesssim \|\nabla^{l+1}\phi\|_{L^6}\|\nabla^{k-l+1}\phi\|_{L^3}  \\
&\displaystyle\overset{\eqref{h20}}\lesssim \|\nabla\phi\|^{1-\f{l+1}{k}}\|\nabla^{k+1}\phi\|^{\f{l+1}{k}}\|\nabla^\alpha\phi\|^{\f{l+1}{k}}\|\nabla^{k+1}\phi\|^{1-\f{l+1}{k}}\\
&\displaystyle\overset{\eqref{h310}}\lesssim M\|\nabla^{k+1}\phi\|,
\end{aligned}$$
where $\alpha$ is defined  by
$$
\frac{k-l}{3}=\lf(\frac{\alpha}{3}-\frac{1}{2}\rg)\f{l+1}{k}+\lf(\frac{k+1}{3}-\frac{1}{2}\rg)\lf(1-\f{l+1}{k}\rg)\rightarrow \alpha=1+\f{3k}{2(l+1)}\in \lf[2, 4\rg].$$
Therefore, we have
$$
I_{7}^4\lesssim  M\lf(
\|\nabla^{k+1} \phi\|^2+
\|\nabla^{k+1} \mathbf{u}\|^2\rg).$$
In a similar way, we obtain $$
I_{7}^5\lesssim  M\lf(
\|\nabla^{k+1} \phi\|^2+
\|\nabla^{k+1} \mathbf{u}\|^2\rg).$$
Substituting the estimates on $I_{7}^i (i=1,\cdots,5)$ into \eqref{hh333}, we have
\begin{equation}\label{hg332}
I_{7}\lesssim  M\lf(\|\nabla^{k}
\sigma\|^2+
\|\nabla^{k+1} \mathbf{u}\|^2+
\|\nabla^{k+1} \phi\|^2\rg).\end{equation}
Substituting \eqref{hg332} into \eqref{3412}, we have \eqref{hr316}.  The proof of Lemma \ref{lem32} is completed.
\end{proof}

\begin{lemma}\label{lem33}  Under the  the  assumption \eqref{h310}, it holds   that
	\begin{equation}
	\label{hr326}
	\begin{aligned}
	\f{d}{dt}&\int \nabla^{k} \mathbf{u}\cdot\nabla^{k+1}\sigma d\mathbf{x}+\|\nabla^{k+1} \sigma\|^2
	\\
	&\lesssim M\lf(
	\|\nabla^{k+2} \mathbf{u}\|^2+\|\nabla^{k+2}\phi\|^2\rg)+\|\nabla^{k+1} \mathbf{u}\|^2,\quad \mathrm{for}\ k=0,1,2.
	\end{aligned}\end{equation}
\end{lemma}
\begin{proof} Multiplying  $\eqref{347}_2$
by $\nabla^{k+1}\sigma$, and using $\eqref{347}_1$ and \eqref{343} , we have
\begin{equation}\label{3413}
\f{d}{dt}\int \nabla^{k} \mathbf{u}\cdot\nabla^{k+1}\sigma d\mathbf{x}+\f{p'(\bar{\rho})}{\bar{\rho}}\|\nabla^{k+1} \sigma\|^2-\bar{\rho}\int({\rm
	div}\nabla^{k}\mathbf{u})^2d\mathbf{x}=I_{8},
\end{equation}where
\begin{equation}\label{hg327}\begin{aligned}
&I_{8}=\underline{\int \nabla^{k}{\rm div}(\sigma\mathbf{u}) {\rm
		div}\nabla^{k}\mathbf{u}d\mathbf{x}}_{I_{8}^1}-\underline{\int \nabla^{k}\lf[
	(\mathbf{u},\nabla)\mathbf{u}-h_1(\sigma)\nabla\sigma\rg] \cdot \nabla^{k+1}\sigma d\mathbf{x}}_{I_{8}^2}\\
&\qquad-\underline{\int \nabla^{k}\lf[h_2(\sigma)\lf(\nu\Delta\mathbf{u}+(\nu+\lambda)\nabla{\rm
		div} \mathbf{u}\rg)\rg] \cdot\nabla^{k+1}\sigma d\mathbf{x}}_{I_{8}^3}\\
&\qquad-\epsilon\underline{\int \nabla^{k}\lf[\nabla\phi\Delta\phi\rg]\cdot\nabla^{k+1}\sigma d\mathbf{x}}_{I_{8}^4}
+\underline{\epsilon\int \nabla^k \big[h_2(\sigma)\nabla\phi\Delta\phi\big]\cdot \nabla^k\sigma d\mathbf{x}}_{I_{8}^5}.
\end{aligned}
\end{equation}
We   estimate  $I_{8}$.
By the same lines as in \cite[Lemma 2.2]{W 12}, we can derive that
$$\begin{aligned}
&\displaystyle I_{8}^1\lesssim  M\lf(\|\nabla^{k+1}
\sigma\|^2+
\|\nabla^{k+1} \mathbf{u}\|^2\rg),\\
&\displaystyle I_{8}^2+I_{8}^3\lesssim  M\lf(\|\nabla^{k+1}
\sigma\|^2+
\|\nabla^{k+2} \mathbf{u}\|^2\rg).
\end{aligned}
$$
For $I_{8}^4$,   it is easy to check that
$$
I_{8}^4\lesssim  \|\nabla\phi\|_{L^\infty}\|\nabla^2\phi\|\|\nabla\sigma\|\overset{\eqref{h310}}\lesssim  M\lf(\|\nabla^2\phi\|^2+
\|\nabla \sigma\|^2\rg),$$ for $k=0$.
When $k\geq1$,
using  Leibniz formula yields that
$$
I_{8}^4=\int \nabla^{k}[\nabla\phi \Delta \phi]\cdot \nabla^{k+1}\sigma d\mathbf{x}\lesssim \sum_{0\leq l\leq k} \|\nabla^{l+1}\phi\nabla^{k-l+1}\phi\|\|\nabla^{k+1}\sigma\|.$$
If $l\leq \lf[\f{k+1}{2}\rg]$,   we get
$$
\begin{aligned}
&\|\nabla^{l+1}\phi\nabla^{k-l+1}\phi\|\lesssim \|\nabla^{l+1}\phi\|_{L^3}\|\nabla^{k-l+1}\phi\|_{L^6}  \\
&\overset{\eqref{h20}}\lesssim \|\nabla^\alpha\phi\|^{1-\f{l-1}{k-1}}\|\nabla^{k+2}\phi\|^{\f{l-1}{k-1}}\|\nabla^3\phi\|^{\f{l-1}{k-1}}\|\nabla^{k+2}\phi\|^{1-\f{l-1}{k-1}}
\overset{\eqref{h310}}\lesssim M\|\nabla^{k+2}\phi\|,
\end{aligned}$$
where $\alpha$ is defined  by
$$
\frac{l}{3}=\lf(\frac{\alpha}{3}-\frac{1}{2}\rg)\lf(1-\f{l-1}{k-1}\rg)+\lf(\frac{k+2}{3}-\frac{1}{2}\rg)\f{l-1}{k-1}\rightarrow \alpha=3-\f{k-1}{2(k-l)}\in \lf[2,3\rg].$$
If $\lf[\f{k+1}{2}\rg]+1\leq l\leq k$,   we get
$$
\begin{aligned}
&\|\nabla^{l+1}\phi\nabla^{k-l+1}\phi\|\lesssim \|\nabla^{l+1}\phi\|_{L^6}\|\nabla^{k-l+1}\phi\|_{L^3}  \\
&\overset{\eqref{h20}}\lesssim \|\nabla^2\phi\|^{1-\f{l+1}{k}}\|\nabla^{k+2}\phi\|^{\f{l+1}{k}}\|\nabla^\alpha\phi\|^{\f{l+2}{k}}\|\nabla^{k+2}\phi\|^{1-\f{l+1}{k}}
\overset{\eqref{h310}}\lesssim M\|\nabla^{k+2}\phi\|,
\end{aligned}$$
where $\alpha$ is defined  by
\begin{eqnarray*}
&&\displaystyle\frac{k-l}{3}=\lf(\frac{\alpha}{3}-\frac{1}{2}\rg)\f{l+1}{k}+\lf(\frac{k+2}{3}-\frac{1}{2}\rg)\lf(1-\f{l+1}{k}\rg) \\
&&\displaystyle\Longrightarrow \alpha=2-\f{k}{2(l+1)}\in \lf[1, 2\rg].
\end{eqnarray*}
Therefore, we have
$$
I_{8}^4\lesssim  M\lf(
\|\nabla^{k+2} \phi\|^2+
\|\nabla^{k+1} \sigma\|^2\rg).$$
In a similar way, we obtain $$
I_{8}^5\lesssim  M\lf(
\|\nabla^{k+2} \phi\|^2+
\|\nabla^{k+1} \sigma\|^2\rg).$$
Substituting the estimates on $I_{8}^i (i=1,\cdots,5)$ into \eqref{hg327}, we have
\begin{equation}\label{hh332}
I_{8}\lesssim  M\lf(\|\nabla^{k+1}
\sigma\|^2+
\|\nabla^{k+1} \mathbf{u}\|^2+
\|\nabla^{k+2} \mathbf{u}\|^2+
\|\nabla^{k+2} \phi\|^2\rg).\end{equation}
Substituting \eqref{hh332} into \eqref{3413}, we have \eqref{hr326}.  The proof of Lemma \ref{lem33} is completed.\end{proof}

%\section{Negative Sobolev estimate}
%\setcounter{equation}{0}

%\hspace{2em}This section is devoted to

\vspace{0.2cm} Now in order to obtain the estimate of the decay of the solution over time, we will establish the evolution of  negative Sobolev norms on solutions to the system   \eqref{h17}-\eqref{hh19}.
\begin{lemma}\label{lem41}  Under the  assumption \eqref{h310}, it holds   that for $s\in (0, \f{1}{2}]$, we have
	\begin{equation}
		\label{hr41}\begin{aligned}
			&\f{d}{dt}\lf(\f{p'(\bar{\rho})}{\bar{\rho}^2}\|\Lambda^{-s} \sigma\|^2+\|\Lambda^{-s} \mathbf{u}\|^2+\|\Lambda^{-s} \nabla\phi\|^2+\|\Lambda^{-s} \lf(\phi^2-1\rg)\|^2\rg)\\
			&	\begin{aligned}
			\lesssim& \lf(\|\nabla\sigma\|_{H^2}^2+
			\|\nabla(\mathbf{u}, \nabla \phi)\|_{H^1}^2+\|\nabla(\phi^2-1)\|^2\rg)\times\\
			&\times \lf(\|\Lambda^{-s} \sigma\|+\|\Lambda^{-s} \mathbf{u}\|+\|\Lambda^{-s} \nabla\phi\|+\|\Lambda^{-s} \lf(\phi^2-1\rg)\|\rg)	\end{aligned}
		\end{aligned}
	\end{equation}and  for $s\in (\f{1}{2}, \f{3}{2})$, we have
	\begin{equation}
		\label{hr42}
		\begin{aligned}
			&\f{d}{dt}\lf(\f{p'(\bar{\rho})}{\bar{\rho}^2}\|\Lambda^{-s} \sigma\|^2+\|\Lambda^{-s} \mathbf{u}\|^2+\|\Lambda^{-s} \nabla\phi\|^2+\|\Lambda^{-s} \lf(\phi^2-1\rg)\|^2\rg)\\
			&	\begin{aligned}
			\lesssim &\|(\sigma, \mathbf{u}, \phi^2-1, \nabla\phi)\|^{s-\f{1}{2}}\lf(
			\|\nabla(\sigma, \mathbf{u},  \nabla\phi)\|_{H^1}+\|\nabla(\phi^2-1)\|\rg)^{\f{5}{2}-s}\times\\
			&\times\lf(\|\Lambda^{-s} \sigma\|+\|\Lambda^{-s} \mathbf{u}\|+\|\Lambda^{-s} \nabla\phi\|+\|\Lambda^{-s} \lf(\phi^2-1\rg)\|\rg).	\end{aligned}
	\end{aligned} 	\end{equation}
\end{lemma}
\begin{proof}
Applying $\Lambda^{-s}$ to $\eqref{342}_1$ and  $\eqref{342}_2$, multiplying the resulting identities by $\f{P'(\bar{\rho})}{\bar{\rho}^2}\Lambda^{-s}\sigma$ and $\Lambda^{-s}\mathbf{u}$, respectively, summing up and using \eqref{343}, we deduce that

\begin{eqnarray}
	\label{hr44}
	\begin{aligned}
		&\f{1}{2}\f{d}{dt}\lf(\f{p'(\bar{\rho})}{\bar{\rho}^2}\|\Lambda^{-s} \sigma\|^2+\|\Lambda^{-s} \mathbf{u}\|^2\rg)+\f{\nu}{\bar{\rho}}\|\Lambda^{-s} \nabla\mathbf{u}\|^2+\f{\nu+\lambda}{\bar{\rho}}\|\Lambda^{-s} {\rm div}\mathbf{u}\|^2\\
&=-\epsilon\underline{\int \Lambda^{-s}\lf(\nabla\phi\Delta\phi\rg)\cdot\Lambda^{-s}\mathbf{u} d\mathbf{x}}_{I_{9}}
		-\f{p'(\bar{\rho})}{\bar{\rho}^2}\underline{\int \Lambda^{-s}\lf(\sigma{\rm div}\mathbf{u}+\nabla \sigma\cdot\mathbf{u}\rg)\Lambda^{-s}\sigma d\mathbf{x}}_{I_{10}}\\
		&-\underline{\int \Lambda^{-s}\lf[(\mathbf{u},\nabla)\mathbf{u}-h_1(\sigma)\nabla\sigma-h_2(\sigma)\lf(\nu\Delta\mathbf{u}+(\nu+\lambda)\nabla{\rm
				div} \mathbf{u}\rg)-\epsilon\nabla\phi\Delta\phi\rg]\cdot \Lambda^{-s}\mathbf{u}  d\mathbf{x}}_{I_{11}}.
	\end{aligned}
\end{eqnarray}
Also, we rewrite \eqref{hg330} as
\begin{equation}\label{hr330} \lf(\phi^2-1\rg)_t+ \mathbf{u}\cdot\nabla\lf(\phi^2-1\rg)-\f{\epsilon}{\rho^2}\Delta\lf(\phi^2-1\rg)+\f{2\epsilon}{\rho^2}|\nabla\phi|^2+ 2\f{\phi^2-1}{\epsilon\rho}\phi^2=0.\end{equation}Then, applying $\Lambda^{-s}$ to \eqref{hg330} and \eqref{hr330}, multiplying the resulting identities by $-\Lambda^{-s}\Delta\phi$  and $\Lambda^{-s}\lf(\phi^2-1\rg)$, respectively,  and summing up the resulting equations, we deduce that
\begin{equation}
\label{hh44}
\begin{aligned}
&\f{1}{2}\f{d}{dt}\lf(\|\Lambda^{-s} \nabla\phi\|^2+\|\Lambda^{-s} \lf(\phi^2-1\rg)\|^2\rg)+\f{\epsilon}{\bar{\rho}^2}\|\Lambda^{-s} \Delta\phi\|^2+\f{2\epsilon}{\bar{\rho}^2}\|\Lambda^{-s} \nabla\lf(\phi^2-1\rg)\|^2\\&
=
\underline{-\int \Lambda^{-s}\nabla\lf(\f{\phi^3-\phi}{\epsilon\rho}\rg)\Lambda^{-s}\nabla \phi d\mathbf{x}}_{I_{12}}\underline{-2\int \Lambda^{-s}\lf(\f{\phi^2-1}{\epsilon\rho}\phi^2\rg)\Lambda^{-s}\lf(\phi^2-1\rg) d\mathbf{x}}_{I_{13}}\\&+\underline{\int \Lambda^{-s}\nabla\lf(-\mathbf{u}\cdot\nabla\phi+h_3(\sigma)\Delta\phi\rg)\Lambda^{-s}\nabla \phi d\mathbf{x}}_{I_{14}}
\\
&+\underline{\int \Lambda^{-s}\lf[-\mathbf{u}\cdot\nabla\lf(\phi^2-1\rg)+2h_3(\sigma)\Delta\lf(\phi^2-1\rg)+\f{\epsilon}{\rho^2}|\nabla\phi|^2\rg]\Lambda^{-s}\lf(\phi^2-1\rg) d\mathbf{x}}_{I_{15}},
\end{aligned}
\end{equation}
where
$h_3(\sigma)=\f{\epsilon}{\bar{\rho}^2}-\f{\epsilon}{\rho^2}.$
In order to estimate the nonlinear terms in the right-hand side of \eqref{hr44} and \eqref{hh44}, we shall use the estimate \eqref{hg25}. This forces us to require that $s\in(0,\f{3}{2})$. If $s\in (0, \f{1}{2}]$, then $\f{1}{2}+\f{s}{3}<1$ and $\f{3}{s}\geq 6$. Then,  we have
$$\begin{aligned}
I_{9}&\lesssim \|\Lambda^{-s}\lf(\nabla\phi\Delta\phi\rg)\| \|\Lambda^{-s}\mathbf{u}\|
\\&\overset{\eqref{hg25}}\lesssim \|\nabla\phi\Delta\phi\|_{L^{\f{1}{\f{1}{2}+\f{s}{3}}}}\|\Lambda^{-s}\mathbf{u}\| \lesssim \|\nabla\phi\|_{L^\f{3}{s}}\|\nabla^2\phi\|\|\Lambda^{-s}\mathbf{u}\|
\\&
\overset{\eqref{h20}}\lesssim \|\nabla^2\phi\|^{\f{1}{2}+s}\|\nabla^3\phi\|^{\f{1}{2}-s} \|\nabla^2\phi\|\|\Lambda^{-s}\mathbf{u}\|
\lesssim\lf(\|\nabla^2\phi\|^2+\|\nabla^3\phi\|^2\rg)\|\Lambda^{-s}\mathbf{u}\|.
\end{aligned}$$
Further from this,  by the same arguments as above and in \cite[Section 3]{W 12},  we have
$$
I_{10}+I_{11}\lesssim \lf(\|\nabla
\sigma\|_{H^1}^2+
\|\nabla \mathbf{u}\|_{H^1}^2\rg)\lf(\|\Lambda^{-s} \sigma\|+\|\Lambda^{-s} \mathbf{u}\|\rg), \ \mathrm{for}\  s\in (0, \f{1}{2}].$$
The estimate on $I_{12}$ is more subtle. Next, we rewrite it as
\begin{equation}\label{hr45}
	\begin{aligned}
		&I_{12}= -\f{2}{\epsilon\bar{\rho}}\int|\nabla^{-s}\nabla \phi|^2 d\mathbf{x}+\f{2}{\epsilon}\underline{\int \Lambda^{-s}\lf[\lf( \f{1}{\bar{\rho}}-\f{1}{\rho}\rg)\nabla\phi\rg]\Lambda^{-s}\nabla \phi d\mathbf{x}}_{I_{12}^1}
		\\&-\f{3}{\epsilon}\underline{\int \Lambda^{-s}\lf( \f{\phi^2-1}{\rho}\nabla\phi\rg)\Lambda^{-s}\nabla \phi d\mathbf{x}}_{I_{12}^2}-\f{1}{\epsilon}\underline{\int \Lambda^{-s}\lf[\lf(\phi^3-\phi\rg)\nabla\lf(\f{1}{\rho}\rg)\rg]\Lambda^{-s}\nabla \phi d\mathbf{x}}_{I_{12}^3}.
\end{aligned} \end{equation}
Then, we have
$$\begin{aligned}
I_{12}^1&\lesssim {\Big\|}\Lambda^{-s}\lf[\lf( \f{1}{\rho}-\f{1}{\bar{\rho}}\rg)\nabla\phi\rg]{\Big\|} \|\Lambda^{-s}\nabla\phi\| \overset{\eqref{hg25}}\lesssim {\Big\|}\lf[\lf( \f{1}{\rho}-\f{1}{\bar{\rho}}\rg)\nabla\phi\rg]{\Big\|}_{L^{\f{1}{\f{1}{2}+\f{s}{3}}}}\|\Lambda^{-s}\nabla\phi\|\\
& \lesssim \|\sigma\|_{L^\f{3}{s}}\|\nabla\phi\| \|\Lambda^{-s}\nabla\phi\|\overset{\eqref{h20}, \eqref{hg24}}\lesssim \|\nabla\sigma\|^{\f{1}{2}+s}\|\nabla^2\sigma\|^{\f{1}{2}-s}\|\nabla^2\phi\|^{1-\theta}\|\Lambda^{-s}\nabla\phi\|^{1+\theta}
\\&
\lesssim \lf(\|\nabla\sigma\|+\|\nabla^2\sigma\|\rg) \lf(\|\nabla^2\phi\|+\|\Lambda^{-s}\nabla\phi\|\rg) \|\Lambda^{-s}\nabla\phi\|\\
&\overset{\eqref{h310}}
\lesssim\lf(\|\nabla\sigma\|_{H^1}^2+\|\nabla^2\phi\|^2\rg)\|\Lambda^{-s}\nabla\phi\|+M\|\Lambda^{-s}\nabla\phi\|^2\,\,\,\text{with}\,\,\theta=\f{1}{2+s},
\end{aligned}$$
$$\begin{aligned}
I_{12}^2&\lesssim {\Big\|}\Lambda^{-s}\lf( \f{\phi^2-1}{\rho}\nabla\phi\rg){\Big\|} \|\Lambda^{-s}\nabla\phi\|
\overset{\eqref{hg25}}\lesssim {\Big\|} \f{\phi^2-1}{\rho}\nabla\phi{\Big\|}_{L^{\f{1}{\f{1}{2}+\f{s}{3}}}}\|\Lambda^{-s}\nabla\phi\|\\
& \lesssim \|\phi^2-1\|_{L^\f{3}{s}}\|\nabla\phi\| \|\Lambda^{-s}\nabla\phi\|
\\
& \overset{\eqref{h20}, \eqref{hg24}}\lesssim \|\nabla(\phi^2-1)\|^{\f{1}{2}+s}\|\nabla^2(\phi^2-1)\|^{\f{1}{2}-s} \|\nabla^2\phi\|^{1-\theta}\|\Lambda^{-s}\nabla\phi\|^{1+\theta}
\\&
\lesssim \lf(\|\nabla(\phi^2-1)\|+\|\nabla^2(\phi^2-1)\|\rg) \lf(\|\nabla^2\phi\|+\|\Lambda^{-s}\nabla\phi\|\rg) \|\Lambda^{-s}\nabla\phi\|\\
&\overset{\eqref{h310}}
\lesssim\lf(\|\nabla(\phi^2-1)\|^2+\|\nabla^2\phi\|^2\rg)\|\Lambda^{-s}\nabla\phi\|+M\|\Lambda^{-s}\nabla\phi\|^2\,\,\,\text{with}\,\,\theta=\f{1}{2+s},
\end{aligned}
$$ and
$$\begin{aligned}
I_{12}^3&\lesssim {\Big\|}\Lambda^{-s}\lf((\phi^2-1)\phi\nabla(\f{1}{\rho})\rg){\Big\|} \|\Lambda^{-s}\nabla\phi\|
\overset{\eqref{hg25}}\lesssim {\Big\|}\lf((\phi^2-1)\phi\nabla(\f{1}{\rho})\rg){\Big\|}_{L^{\f{1}{\f{1}{2}+\f{s}{3}}}}\|\Lambda^{-s}\nabla\phi\|\\
&\overset{\eqref{hr33}} \lesssim \|\phi^2-1\|_{L^\f{3}{s}} \|\nabla\sigma\| \|\Lambda^{-s}\nabla\phi\|
\\
& \overset{\eqref{h20}}\lesssim \|\nabla(\phi^2-1)\|^{\f{1}{2}+s}\|\nabla^2(\phi^2-1)\|^{\f{1}{2}-s} \|\nabla\sigma\|\|\Lambda^{-s}\nabla\phi\|
\\&
\lesssim \lf(\|\nabla(\phi^2-1)\|+\|\nabla^2(\phi^2-1)\|\rg)\|\nabla\sigma\|\|\Lambda^{-s}\nabla\phi\|\\
&\overset{\eqref{h310}}
\lesssim\lf(\|\nabla(\phi^2-1)\|^2+\|\nabla^2\phi\|^2+\|\nabla\sigma\|^2\rg)\|\Lambda^{-s}\nabla\phi\|.
\end{aligned}$$  Therefore, we obtain from  \eqref{hr45} that
$$I_{12}+\f{1}{\epsilon\bar{\rho}}\|\Lambda^{-s}\nabla \phi\|^2 \lesssim\lf(\|\nabla\sigma\|_{H^1}^2+\|\nabla(\phi^2-1)\|^2+\|\nabla^2\phi\|^2\rg)\|\Lambda^{-s}\nabla\phi\|.$$
Similarly, for $I_{13}$, we rewrite it as
\begin{equation}\label{hh45}
\begin{aligned}
I_{13}=& -\f{2}{\epsilon\bar{\rho}}\int|\Lambda^{-s}\lf(\phi^2-1\rg)|^2 d\mathbf{x}+\f{2}{\epsilon}\underline{\int \Lambda^{-s}\lf[\lf( \f{1}{\bar{\rho}}-\f{1}{\rho}\rg)\lf(\phi^2-1\rg)\rg]\Lambda^{-s}\lf(\phi^2-1\rg) d\mathbf{x}}_{I_{13}^1}
\\&-\f{2}{\epsilon}\underline{\int \Lambda^{-s}\lf( \f{\lf(\phi^2-1\rg)^2}{\rho}\rg)\Lambda^{-s}\lf(\phi^2-1\rg) d\mathbf{x}}_{I_{13}^2}.
\end{aligned} \end{equation}
Then, we have

$$\begin{aligned}
I_{13}^1&\lesssim {\Big\|}\Lambda^{-s}\lf[\lf( \f{1}{\rho}-\f{1}{\bar{\rho}}\rg)\lf(\phi^2-1\rg)\rg]{\Big\|} \|\Lambda^{-s}\lf(\phi^2-1\rg)\|\\
& \lesssim \|\sigma\|_{L^\f{3}{s}}\|\phi^2-1\| \|\Lambda^{-s}\lf(\phi^2-1\rg)\|\\
&\overset{\eqref{h20}, \eqref{hg24}}\lesssim \|\nabla\sigma\|^{\f{1}{2}+s}\|\nabla^2\sigma\|^{\f{1}{2}-s}\|\nabla\lf(\phi^2-1\rg)\|^{1-\theta}\|\Lambda^{-s}\lf(\phi^2-1\rg)\|^{1+\theta}
\\&
\lesssim \lf(\|\nabla\sigma\|+\|\nabla^2\sigma\|\rg) \lf(\|\nabla\lf(\phi^2-1\rg)\|+\|\Lambda^{-s}\lf(\phi^2-1\rg)\|\rg) \|\Lambda^{-s}\lf(\phi^2-1\rg)\|\\
&\overset{\eqref{h310}}
\lesssim\lf(\|\nabla\sigma\|_{H^1}^2+\|\nabla\lf(\phi^2-1\rg)\|^2\rg)\|\Lambda^{-s}\lf(\phi^2-1\rg)\|+M\|\Lambda^{-s}\lf(\phi^2-1\rg)\|^2,
\end{aligned}$$ and
$$\begin{aligned}
&\displaystyle I_{13}^2\lesssim {\Big\|}\Lambda^{-s}\lf( \f{\lf(\phi^2-1\rg)^2}{\rho}\rg){\Big\|} \|\Lambda^{-s}\lf(\phi^2-1\rg)\|
 \lesssim \|\phi^2-1\|_{L^\f{3}{s}}\|\phi^2-1\| \|\Lambda^{-s}\lf(\phi^2-1\rg)\|
\\
& \displaystyle\overset{\eqref{h20}, \eqref{hg24}}\lesssim \|\nabla(\phi^2-1)\|^{\f{1}{2}+s}\|\nabla^2(\phi^2-1)\|^{\f{1}{2}-s} \|\nabla\lf(\phi^2-1\rg)\|^{1-\theta}\|\Lambda^{-s}\lf(\phi^2-1\rg)\|^{1+\theta}
\\&\displaystyle
\lesssim \lf(\|\nabla(\phi^2-1)\|+\|\nabla^2(\phi^2-1)\|\rg) \lf(\|\nabla\lf(\phi^2-1\rg)\|+\|\Lambda^{-s}\lf(\phi^2-1\rg)\|\rg) \|\Lambda^{-s}\lf(\phi^2-1\rg)\|\\
&\displaystyle\overset{\eqref{h310}}
\lesssim\lf(\|\nabla(\phi^2-1)\|^2+\|\nabla^2\phi\|^2\rg)\|\Lambda^{-s}\lf(\phi^2-1\rg)\|+M\|\Lambda^{-s}\lf(\phi^2-1\rg)\|^2,\ \mathrm{with}  \ \theta=\f{1}{2+s}.
\end{aligned}
$$Therefore, we obtain from  \eqref{hh45} that
$$
I_{13}+\f{2}{\epsilon\bar{\rho}}\|\Lambda^{-s}\lf(\phi^2-1\rg)\|^2
\lesssim\lf(\|\nabla\sigma\|_{H^1}^2+\|\nabla(\phi^2-1)\|^2+\|\nabla^2\phi\|^2\rg)\|\Lambda^{-s}\lf(\phi^2-1\rg)\|.
$$
For $I_{14}$, we can prove that
$$\begin{aligned}
&\displaystyle I_{14}\lesssim \|\Lambda^{-s}\nabla\lf(-\mathbf{u}\cdot\nabla\phi+h_3(\sigma)\Delta\phi\rg)\| \|\Lambda^{-s}\nabla\phi\|\\
&\displaystyle\overset{\eqref{hg25}}\lesssim \|\nabla\lf(-\mathbf{u}\cdot\nabla\phi+h_3(\sigma)\Delta\phi\rg)\|_{L^{\f{1}{\f{1}{2}+\f{s}{3}}}}\|\Lambda^{-s}\nabla\phi\|\\
&\displaystyle\overset{\eqref{hh20}} \lesssim\lf(\|\mathbf{u}\|_{L^\f{3}{s}} \|\nabla^2\phi\|+\|\nabla\phi\|_{L^\f{3}{s}} \|\nabla\mathbf{u}\|+\|\sigma\|_{L^\f{3}{s}} \|\nabla^3\phi\|+\|\nabla\sigma\|_{L^\f{3}{s}} \|\nabla^2\phi\|\rg)  \|\Lambda^{-s}\nabla\phi\|\\
&\displaystyle\lesssim\lf(\|\mathbf{u}\|_{L^\f{3}{s}} +\|\nabla\phi\|_{L^\f{3}{s}} +\|\sigma\|_{L^\f{3}{s}} +\|\nabla\sigma\|_{L^\f{3}{s}} \rg) (\|\nabla\mathbf{u}\|+\|\nabla^2\phi\|_{H^1}) \|\Lambda^{-s}\nabla\phi\|\\
&\displaystyle\overset{\eqref{h20}}\lesssim(\|\nabla\mathbf{u}\|^{\f{1}{2}+s}\|\nabla^2\mathbf{u}\|^{\f{1}{2}-s}+\|\nabla^2\phi\|^{\f{1}{2}+s}\|\nabla^3\phi\|^{\f{1}{2}-s}+\|\nabla\sigma\|^{\f{1}{2}+s}\|\nabla^2\sigma\|^{\f{1}{2}-s}\\
&\displaystyle\qquad+\|\nabla^2\sigma\|^{\f{1}{2}+s}\|\nabla^3\sigma\|^{\f{1}{2}-s})(\|\nabla\mathbf{u}\|+\|\nabla^2\phi\|_{H^1}) \|\Lambda^{-s}\nabla\phi\|\\
&\displaystyle \lesssim\lf(\|\nabla\mathbf{u}\|_{H^1}^2+\|\nabla^2\phi\|_{H^1}^2+\|\nabla\sigma\|_{H^2}^2\rg)\|\Lambda^{-s}\nabla\phi\|.
\end{aligned}$$
Last, for $I_{15}$, we get
$$\begin{aligned}
&\displaystyle I_{15}\lesssim {\Big \|}\Lambda^{-s}\lf[-\mathbf{u}\cdot\nabla\lf(\phi^2-1\rg)+2h_3(\sigma)\Delta\lf(\phi^2-1\rg)+\f{2\epsilon}{\rho^2}|\nabla\phi|^2\rg]{\Big \|} \|\Lambda^{-s}\lf(\phi^2-1\rg)\|\\
&\displaystyle\overset{\eqref{hg25}}\lesssim {\Big \|}-\mathbf{u}\cdot\nabla\lf(\phi^2-1\rg)+2h_3(\sigma)\Delta\lf(\phi^2-1\rg)+\f{2\epsilon}{\rho^2}|\nabla\phi|^2{\Big \|}_{L^{\f{1}{\f{1}{2}+\f{s}{3}}}}\|\Lambda^{-s}\lf(\phi^2-1\rg)\|\\
&\displaystyle \lesssim\lf(\|\mathbf{u}\|_{L^\f{3}{s}} \|\nabla\lf(\phi^2-1\rg)\|+\|\sigma\|_{L^\f{3}{s}} \|\nabla^2\lf(\phi^2-1\rg)\|+\|\nabla\phi\|_{L^\f{3}{s}} \|\nabla\phi\|\rg)  \|\Lambda^{-s}\lf(\phi^2-1\rg)\|\\
&\displaystyle \begin{aligned} \overset{\eqref{h20}}\lesssim &\lf(\|\nabla\mathbf{u}\|^{\f{1}{2}+s}\|\nabla^2\mathbf{u}\|^{\f{1}{2}-s}\|\nabla\lf(\phi^2-1\rg)\|+\|\nabla\sigma\|^{\f{1}{2}+s}\|\nabla^2\sigma\|^{\f{1}{2}-s}\|\nabla^2\phi\|\rg) \|\Lambda^{-s}\lf(\phi^2-1\rg)\|
\\&\displaystyle\overset{\eqref{hg24}}
+\|\nabla^2\phi\|^{\f{1}{2}+s}\|\nabla^3\phi\|^{\f{1}{2}-s}\|\nabla^2\phi\|^{1-\theta} \|\Lambda^{-s}\nabla\phi\|^\theta\|\Lambda^{-s}\lf(\phi^2-1\rg)\| \end{aligned}\\
&\displaystyle\begin{aligned}
\lesssim&\lf(\|\nabla\mathbf{u}\|_{H^1}^2+\|\nabla\lf(\phi^2-1\rg)\|^2+\|\nabla^2\phi\|_{H^1}^2+\|\nabla\sigma\|_{H^1}^2\rg)\|\Lambda^{-s}\lf(\phi^2-1\rg)\|\\
&\displaystyle+M\lf(\|\Lambda^{-s}\nabla\phi\|^2+\|\Lambda^{-s}\lf(\phi^2-1\rg)\|^2\rg),\ \mathrm{with}\  \theta=\f{1}{2+s}.\end{aligned}
\end{aligned}$$
Substituting the estimates on $I_{i} (i=9, \cdots, 15)$  into \eqref{hr44} and \eqref{hh44}, respectively, we obtain \eqref{hr41}.

\vspace{0.2cm} Next, we derive \eqref{hr42}. To this end, for $s\in (\f{1}{2}, \f{3}{2})$,  we shall estimate the right-hand sides of \eqref{hr44} and \eqref{hh44} in a different way. Since $s\in (\f{1}{2}, \f{3}{2})$, we have that  $\f{1}{2}+\f{s}{3}<1$ and $2<\f{3}{s}<6$.
Then,  we have
$$\begin{aligned}
I_{9}&\lesssim \cdots\cdots
\lesssim \|\nabla\phi\|_{L^\f{3}{s}}\|\nabla^2\phi\|\|\Lambda^{-s}\mathbf{u}\|
\\
&
\overset{\eqref{h20}}\lesssim \|\nabla\phi\|^{s-\f{1}{2}}\|\nabla^2\phi\|^{\f{5}{2}-s} \|\Lambda^{-s}\mathbf{u}\|.
\end{aligned}$$
Also, for $s\in (\f{1}{2}, \f{3}{2})$, by the same arguments as in \cite[Section 3]{W 12},  we have
$$
I_{10}+I_{11}\lesssim \|(\sigma, \mathbf{u})\|^{s-\f{1}{2}} \|\nabla(
\sigma, \mathbf{u})\|_{H^1}^{\f{5}{2}-s}\lf(\|\Lambda^{-s} \sigma\|+\|\Lambda^{-s} \mathbf{u}\|\rg).
$$
For the estimate on $I_{12}$, using \eqref{hr45}, we have
$$\begin{aligned}
I_{12}^1&\lesssim \cdots\cdots\lesssim \|\sigma\|_{L^\f{3}{s}}\|\nabla\phi\| \|\Lambda^{-s}\nabla\phi\|
\\&
\overset{\eqref{h20}, \eqref{hg24}}\lesssim \|\sigma\|^{s-\f{1}{2}}\|\nabla\sigma\|^{\f{3}{2}-s}\|\nabla^2\phi\|^{1-\theta}\|\Lambda^{-s}\nabla\phi\|^{1+\theta}\\
&
\lesssim\|\sigma\|^{s-\f{1}{2}}\|\nabla\sigma\|^{\f{3}{2}-s}\lf(\|\nabla^2\phi\|+\|\Lambda^{-s}\nabla\phi\|\rg) \|\Lambda^{-s}\nabla\phi\|\\
&\overset{\eqref{h310}}
\lesssim\|\sigma\|^{s-\f{1}{2}}\|\nabla\sigma\|^{\f{3}{2}-s}\|\nabla^2\phi\|\|\Lambda^{-s}\nabla\phi\|+M\|\Lambda^{-s}\nabla\phi\|^2, \,\,\,\text{with}\,\,\theta=\f{1}{2+s},
\end{aligned}$$
$$\begin{aligned}
I_{12}^2&\lesssim \cdots\cdots\lesssim \|\phi^2-1\|_{L^\f{3}{s}}\|\nabla\phi\| \|\Lambda^{-s}\nabla\phi\|
\\
& \overset{\eqref{h20}, \eqref{hg24}}\lesssim \|\phi^2-1\|^{s-\f{1}{2}}\|\nabla(\phi^2-1)\|^{\f{3}{2}-s} \|\nabla^2\phi\|^{1-\theta}\|\Lambda^{-s}\nabla\phi\|^{1+\theta}
\\&
\lesssim \|\phi^2-1\|^{s-\f{1}{2}}\|\nabla(\phi^2-1)\|^{\f{3}{2}-s} \lf(\|\nabla^2\phi\|+\|\Lambda^{-s}\nabla\phi\|\rg) \|\Lambda^{-s}\nabla\phi\|\\
&\overset{\eqref{h310}}
\lesssim\|\phi^2-1\|^{s-\f{1}{2}}\|\nabla(\phi^2-1)\|^{\f{3}{2}-s}\|\nabla^2\phi\|\|\Lambda^{-s}\nabla\phi\|+M\|\Lambda^{-s}\nabla\phi\|^2,\,\,\,\text{with}\,\,\theta=\f{1}{2+s},
\end{aligned}
$$ and
$$\begin{aligned}
I_{12}^3&\lesssim \cdots\cdots \lesssim \|\phi^2-1\|_{L^\f{d}{3}} \|\nabla\sigma\| \|\Lambda^{-s}\nabla\phi\|
\\
& \overset{\eqref{h20}}\lesssim \|\phi^2-1\|^{s-\f{1}{2}}\|\nabla(\phi^2-1)\|^{\f{3}{2}-s} \|\nabla\sigma\|\|\Lambda^{-s}\nabla\phi\|.
\end{aligned}$$  Therefore, we obtain from  \eqref{hr45} that
$$I_{12}+\f{1}{\epsilon\bar{\rho}}\|\Lambda^{-s}\nabla \phi\|^2\lesssim\|(\sigma, \phi^2-1)\|^{s-\f{1}{2}}\|\nabla(\sigma, \phi^2-1, \nabla\phi)\|^{\f{3}{2}-s}\|\Lambda^{-s}\nabla\phi\|.$$
Similarly, for the estimate on $I_{13}$, using \eqref{hh45}, we have
$$\begin{aligned}
&\displaystyle I_{13}^1\lesssim\cdots\cdots \lesssim \|\sigma\|_{L^\f{3}{s}}\|\phi^2-1\| \|\Lambda^{-s}\lf(\phi^2-1\rg)\|\\
&\displaystyle\overset{\eqref{h20}, \eqref{hg24}}\lesssim \|\sigma\|^{s-\f{1}{2}}\|\nabla\sigma\|^{\f{3}{2}-s}\|\nabla\lf(\phi^2-1\rg)\|^{1-\theta}\|\Lambda^{-s}\lf(\phi^2-1\rg)\|^{1+\theta}
\\&\displaystyle
\lesssim \|\sigma\|^{s-\f{1}{2}}\|\nabla\sigma\|^{\f{3}{2}-s} \lf(\|\nabla\lf(\phi^2-1\rg)\|+\|\Lambda^{-s}\lf(\phi^2-1\rg)\|\rg) \|\Lambda^{-s}\lf(\phi^2-1\rg)\|\\
&\displaystyle\overset{\eqref{h310}}
\lesssim\|\sigma\|^{s-\f{1}{2}}\|\nabla\sigma\|^{\f{3}{2}-s}\|\nabla\lf(\phi^2-1\rg)\|^\|\Lambda^{-s}\lf(\phi^2-1\rg)\|+M\|\Lambda^{-s}\lf(\phi^2-1\rg)\|^2,
\end{aligned}$$ and
$$\begin{aligned}
&\displaystyle I_{13}^2\lesssim \cdots\cdots\lesssim \|\phi^2-1\|_{L^\f{3}{s}}\|\phi^2-1\| \|\Lambda^{-s}\lf(\phi^2-1\rg)\|
\\
&\displaystyle \overset{\eqref{h20}, \eqref{hg24}}\lesssim \|\phi^2-1\|^{s-\f{1}{2}}\|\nabla(\phi^2-1)\|^{\f{3}{2}-s} \|\nabla\lf(\phi^2-1\rg)\|^{1-\theta}\|\Lambda^{-s}\lf(\phi^2-1\rg)\|^{1+\theta}
\\&\displaystyle
\lesssim \|\phi^2-1\|^{s-\f{1}{2}}\|\nabla(\phi^2-1)\|^{\f{3}{2}-s} \lf(\|\nabla\lf(\phi^2-1\rg)\|+\|\Lambda^{-s}\lf(\phi^2-1\rg)\|\rg) \|\Lambda^{-s}\lf(\phi^2-1\rg)\|\\
&\displaystyle\overset{\eqref{h310}}
\lesssim\|\phi^2-1\|^{s-\f{1}{2}}\|\nabla(\phi^2-1)\|^{\f{5}{2}-s}\|\Lambda^{-s}\lf(\phi^2-1\rg)\|+M\|\Lambda^{-s}\lf(\phi^2-1\rg)\|^2,\ \mathrm{with}  \ \theta=\f{1}{2+s}.
\end{aligned}
$$ Therefore, we obtain from  \eqref{hh45} that
$$
I_{13}+\f{2}{\epsilon\bar{\rho}}\|\Lambda^{-s}\lf(\phi^2-1\rg)\|^2
\lesssim\|(\sigma, \phi^2-1)\|^{s-\f{1}{2}}\|\nabla(\sigma, \phi^2-1)\|^{\f{3}{2}-s}\|\Lambda^{-s}\lf(\phi^2-1\rg)\|.
$$
For $I_{14}$,
we can prove that
$$\begin{aligned}
I_{14}&\lesssim\cdots\cdots
\\
&\lesssim\lf(\|\mathbf{u}\|_{L^\f{3}{s}} +\|\nabla\phi\|_{L^\f{3}{s}} +\|\sigma\|_{L^\f{3}{s}} +\|\nabla\sigma\|_{L^\f{3}{s}} \rg) (\|\nabla\mathbf{u}\|+\|\nabla^2\phi\|_{H^1}) \|\Lambda^{-s}\nabla\phi\|\\
& \begin{aligned} \overset{\eqref{h20}}\lesssim &(\|\mathbf{u}\|^{s-\f{1}{2}}\|\nabla\mathbf{u}\|^{\f{3}{2}-s}+\|\nabla\phi\|^{s-\f{1}{2}}\|\nabla^2\phi\|^{\f{3}{2}-s}+\|\sigma\|^{s-\f{1}{2}}\|\nabla\sigma\|^{\f{3}{2}-s}
\\&
+\|\nabla\sigma\|^{s-\f{1}{2}}\|\nabla^2\sigma\|^{\f{3}{2}-s})(\|\nabla\mathbf{u}\|+\|\nabla^2\phi\|_{H^1}) \|\Lambda^{-s}\nabla\phi\|\end{aligned}\\
&\begin{aligned}
\lesssim&\|(\sigma, \mathbf{u}, \nabla\phi)\|^{s-\f{1}{2}}\|\nabla(\sigma, \mathbf{u}, \nabla\phi)\|^{\f{3}{2}-s}(\|\nabla\mathbf{u}\|+\|\nabla^2\phi\|_{H^1})\|\Lambda^{-s}\nabla\phi\|\\
&+\lf(\|\nabla\sigma\|^2_{H^1}+\|\nabla\mathbf{u}\|^2+\|\nabla^2\phi\|^2_{H^1}\rg)\|\Lambda^{-s}\nabla\phi\|.\end{aligned}
\end{aligned}$$
Last, for $I_{15}$, we get
$$\begin{aligned}
I_{15}&\displaystyle\lesssim\cdots\cdots\\
& \displaystyle\lesssim\lf(\|\mathbf{u}\|_{L^\f{3}{s}} \|\nabla\lf(\phi^2-1\rg)\|+\|\sigma\|_{L^\f{3}{s}} \|\nabla^2\lf(\phi^2-1\rg)\|+\|\nabla\phi\|_{L^\f{3}{s}} \|\nabla\phi\|\rg)  \|\Lambda^{-s}\lf(\phi^2-1\rg)\|
\\
& \displaystyle\begin{aligned} \overset{\eqref{h20}}\lesssim &\lf(\|\mathbf{u}\|^{s-\f{1}{2}}\|\nabla\mathbf{u}\|^{\f{3}{2}-s}\|\nabla\lf(\phi^2-1\rg)\|+\|\sigma\|^{s-\f{1}{2}}\|\nabla\sigma\|^{\f{3}{2}-s}\|\nabla^2\phi\|\rg) \|\Lambda^{-s}\lf(\phi^2-1\rg)\|
\\&\displaystyle\overset{\eqref{hg24}}
+\|\nabla\phi\|^{s-\f{1}{2}}\|\nabla^2\phi\|^{\f{3}{2}-s}\|\nabla^2\phi\|^{1-\theta} \|\Lambda^{-s}\nabla\phi\|^\theta\|\Lambda^{-s}\lf(\phi^2-1\rg)\| \end{aligned}\\
&\displaystyle\lesssim\lf(\|\mathbf{u}\|^{s-\f{1}{2}}\|\nabla\mathbf{u}\|^{\f{3}{2}-s}+\|\sigma\|^{s-\f{1}{2}}\|\nabla\sigma\|^{\f{3}{2}-s}+\|\nabla\phi\|^{s-\f{1}{2}}\|\nabla^2\phi\|^{\f{3}{2}-s}\rg)\times\\
&\displaystyle\times \lf(\|\nabla\lf(\phi^2-1\rg)\|+\|\nabla^2\phi\|\rg)\|\Lambda^{-s}\lf(\phi^2-1\rg)\|+M\lf(\|\Lambda^{-s}\nabla\phi\|^2+\|\Lambda^{-s}\lf(\phi^2-1\rg)\|^2\rg),\end{aligned}
$$ with  $\displaystyle\theta=\f{1}{2+s}$.
Substituting the estimates on $I_{i} (i=9, \cdots, 15)$  into \eqref{hr44} and \eqref{hh44}, respectively, we obtain \eqref{hr42}.  The proof of Lemma \ref{lem41} is completed.\end{proof}

\section{The proof of Theorem \ref{theo11}}
\setcounter{equation}{0}

\hspace{2em}In this section, we shall combine all the estimates that we have derived in the previous
section and the Sobolev interpolation to prove Theorem \ref{theo11}.

\subsection{The a priori estimate and the existence for global solutions}\label{subs1}

\hspace{2em}In order to extend the local solution to global solution continuously,  we first give a prior estimate by using the series Lemmas in section 2.
\begin{proposition}  \label{estimate} \textbf{(a priori estimate).}
There is positive constant $m_1>0$ and $M_1>0$,  then, if  $(\rho,	\mathbf{u}, \phi)\in X_{m_1,M_1}\big([0,T^*]\big)$
\begin{equation}\label{a priori estimate}
\begin{aligned}
&\sup_{t\in[0,T]}\Big\{\|(\rho-\bar{\rho}, \mathbf{u})(t)\|^2_{H^3}+ \|\nabla \phi(t)\|^2_{H^2}+ \|\phi^2(t)-1\|^2\Big\}\\
&\quad+\int_0^{T}\|\nabla\rho\|^2_{H^2} d\tau+ \int_0^{+\infty}\|\nabla\mathbf{u},\nabla\phi\|^2_{H^3} d\tau\\
&\displaystyle\quad\lesssim \|\rho_0-\bar{\rho}\|^2_{H^3}+\|\mathbf{u}_0\|^2_{H^3}+ \|\nabla \phi_0\|^2_{H^2}+\|\phi^2_0-1\|^2.
\end{aligned}
\end{equation}
\end{proposition}
\begin{proof}We first close the energy estimates at each $l-$th level  to prove \eqref{h18}.
%For $N\geq 3,$
Let $ 0\leq l\leq 2$. Summing up the estimate \eqref{hr329} of Lemma \ref{lem34} for from $k=l$ to $2$, we obtain
\begin{equation}
	\label{h51}\begin{aligned}
	&\f{d}{dt}\sum_{1\leq l\leq k\leq 2}\|\nabla^{k+1} \phi\|^2+\sum_{1\leq l\leq k\leq 2}\|\nabla^{k+2}\phi\|^2\\&	
	\lesssim M\sum_{1\leq l\leq k\leq 2}\lf(\|\nabla^{k+1}\sigma\|^2+ \|\nabla^{k+1}\phi\|^2+	\|\nabla^{k+2} \mathbf{u}\|^2\rg).
	\end{aligned}
\end{equation}
Also, summing up the estimate \eqref{hr316} of Lemma \ref{lem32} for from $k=l$ to $3$, we obtain
\begin{equation}
	\label{h52}
	\begin{aligned}
		&\f{d}{dt}\sum_{0\leq l\leq k\leq 3}\lf(\|\nabla^k\mathbf{u}\|^2
		+\f{p'(\bar{\rho})}{\bar{\rho}^2}\|\nabla^k\sigma\|^2\rg)+\sum_{0\leq l\leq k\leq 3}\|\nabla^{k+1} \mathbf{u}\|^2\\
		&\lesssim M\sum_{0\leq l\leq k\leq 3}\lf(\|\nabla^{k}
		\sigma\|^2+\|\nabla^{k+1}\phi\|^2\rg).
\end{aligned}\end{equation}
Last, summing up the estimate \eqref{hr326} of Lemma \ref{lem33} for from $k=l$ to $2$, we obtain
\begin{equation}
	\label{h53}
	\begin{aligned}
		&\f{d}{dt}\sum_{0\leq l\leq k\leq 2}\int \nabla^{k} \mathbf{u}\cdot\nabla^{k+1}\sigma d\mathbf{x}+\sum_{0\leq l\leq k\leq 2} \|\nabla^{k+1} \sigma\|^2
		\\
		&\leq  M \sum_{0\leq l\leq k\leq 2}\lf(	\|\nabla^{k+2} \mathbf{u}\|^2+\|\nabla^{k+2}\phi\|^2\rg)+\sum_{0\leq l\leq k\leq 2}\|\nabla^{k+1} \mathbf{u}\|^2.
\end{aligned}\end{equation}
Let $\eta\in (0,1]$ be suitably small. Then, summing \eqref{hr32}, \eqref{h51}, \eqref{h52} and $\eta\cdot \eqref{h53}$, and choosing $M>0$  small enough, we obtain
\begin{equation}
	\label{h54}
	\f{d}{dt}\mathcal{E}_l(t)+\Lambda_{l}(t)\leq 0,
\end{equation}
%for any  $0\leq l\leq 3$,
where
\begin{equation}\label{h55}\begin{aligned}
			\mathcal{E}_l(t)&\overset{def}{=}\sum_{0\leq l\leq k\leq  3}\lf(\|\nabla^k\mathbf{u}\|^2
			+\f{p'(\bar{\rho})}{\bar{\rho}^2}\|\nabla^k\sigma\|^2\rg)+\eta \sum_{0\leq l\leq k\leq 2}\int \nabla^{k} \mathbf{u}\cdot\nabla^{k+1}\sigma d\mathbf{x}
			\\
			&+\sum_{1\leq l\leq k\leq 2}\|\nabla^{k+1} \phi\|^2+\int\Big(\rho
			\mathbf{u}^2+|\rho-\bar\rho|^2+|\nabla
			\phi|^2+(\phi^2-1)^2\Big)d\mathbf{x},\end{aligned}
	\end{equation}
	and
	\begin{equation}\label{hh55}
	\begin{aligned}
\Lambda_{l}(t)&\overset{def}{=}\eta\sum_{1\leq l+1\leq k\leq 3} \|\nabla^{k} \sigma\|^2+\lf(1-M\rg)\sum_{2\leq l+1\leq k\leq 3}\|\nabla^{k+1}\phi\|^2\\
&+\lf(1-\eta\rg)\sum_{0\leq l\leq k\leq 3}\|\nabla^{k+1} \mathbf{u}\|^2+\|\nabla\mathbf{u}\|^2+\|
\mu \|^2\\
&+\|\Delta\phi\|^2-M\|\nabla^2\phi\|^2+\|\nabla\lf(\phi^2-1\rg)\|^2.\end{aligned}
\end{equation}
Notice that since $\eta\in (0, 1]$ and $M>0$ are suitably small, we obtain from \eqref{h55} and \eqref{hh55} that
\begin{equation}\label{h56}\begin{aligned}
		&\mathcal{E}_l^3(t) \backsimeq \|\nabla^l(\sigma, \mathbf{u})(t)\|^2_{H^{2}}+\|\nabla^{l+1} \phi(t)\|^2_{H^{2-l}}+\|\phi^2-1\|^2,\\
		&\begin{aligned}\Lambda_{l}^3(t)\backsimeq& \|\nabla ^{l+1} \sigma(t)\|^2_{H^{2-l}}+ \|\nabla^{l+1}\mathbf{u}(t)\|^2_{H^{3-l}}+ \|\nabla^{l+1} \phi(t)\|^2_{H^{3-l}}+\|\mu \|^2,\end{aligned}
\end{aligned} \end{equation} uniformly for all $t\geq 0$.
Now taking $l=0$ and $n=3$ in \eqref{h54}, and then  using \eqref{h56}, we get
\begin{equation}\label{h57}
	\begin{aligned}
		&\|\sigma(t)\|^2_{H^{3}}+\|\mathbf{u}(t)\|^2_{H^{3}}+ \|\nabla \phi(t)\|^2_{H^{2}}+\|\phi^2(t)-1\|^2\\%\lesssim \mathcal{E}_0^3(t)\\&\lesssim \mathcal{E}_0^3(0)
&\lesssim \|\rho_0-\bar{\rho}\|^2_{H^{3}}+\|\mathbf{u}_0\|^2_{H^{3}}+ \|\nabla \phi_0\|^2_{H^{2}}+\|\phi^2_0-1\|^2.
	\end{aligned}
\end{equation}
Using \eqref{g17} and \eqref{h57}, by a standard continuity argument, we can close the a priori estimate
\eqref{h310}.  This in turn allows us to take $l=0$ % and $n=3$
 in \eqref{h54}, and then integrate it directly in time to obtain  \eqref{a priori estimate}.
\end{proof}
%Now, the proof of  existence and  uniqueness
%of global solution is standard. So, we omit the details for brevity.

\vspace{0.2cm} From the a priori estimate  \eqref{estimate} and the local existence of the solution $(\rho,\mathbf{u}, \phi)$  for the Cauchy problem \eqref{h17}-\eqref{hh19}(see Proposition \eqref{pro311}), we can construct a solution for $t\in[0,T^*]$, which satisfies \eqref{h315} in $\mathbb{R}^3\times[0,T^*]$, $T^*$ only depends on the initial data  of the Cauchy problem \eqref{h17}-\eqref{hh19}. From \eqref{h57},  one can start again from $T^*$, by the same way, one can find a solution in $[T^*,2T^*]$, and so on. Thus the existence and uniqueness of the global solution is obtained. Meanwhile, by using maximum principle, $-1\leq\phi\leq1$.

\subsection{Decay rate}\label{subs2}

\hspace{2em} We  first prove \eqref{h124}-\eqref{hg124} for $s\in (0, \f{1}{2}]$. Define $$\mathcal{E}_{-s}:=\f{p'(\bar{\rho})}{\bar{\rho}^2}\|\Lambda^{-s} \sigma\|^2+\|\Lambda^{-s} \mathbf{u}\|^2+\|\Lambda^{-s} \nabla\phi\|^2+\|\Lambda^{-s} (\phi^2-1)\|^2. $$
Then, integrating in time \eqref{hr41} and using  \eqref{h18}, we obtain that
$$\begin{aligned}
\mathcal{E}_{-s}(t)&\lesssim \mathcal{E}_{-s}(0)+ \int_0^t \lf(\|\nabla\sigma\|_{H^2}^2+
\|\nabla(\mathbf{u}, \nabla \phi)\|_{H^1}^2+\|\nabla(\phi^2-1)\|^2\rg)\sqrt{\mathcal{E}_{-s}(\tau)}d\tau\\
&\lesssim 1+ \sup_{0\leq \tau\leq t}\sqrt{\mathcal{E}_{-s}(\tau)}, \ \mathrm{for}\  s\in (0, \f{1}{2}],
\end{aligned}$$ which implies  \eqref{h124}.
For the proof of  \eqref{hg124}, we need the following lemma.

\begin{lemma}\label{lem51}  Under the  the  assumption \eqref{h310}, it holds   that
	\begin{equation}
		\label{hh58}\begin{aligned}
		&\f{1}{2}\f{d}{dt}\|\nabla^{k}\lf(\phi^2-1\rg)\|^2+\f{\epsilon}{8\bar{\rho}^2}\|\nabla^{k+1}\lf(\phi^2-1\rg)\|^2+\f{1}{2\bar{\rho}\epsilon}\|\nabla^{k}\lf(\phi^2-1\rg)\|^2
\\&	\lesssim M\lf(\|\nabla^{k+1}
		\sigma\|^2+
		\|\nabla^{k+2} \phi\|^2+
		\|\nabla^{k+1} \mathbf{u}\|^2\rg),\ \mathrm{for}\  k=1, 2.		\end{aligned}
	\end{equation}
\end{lemma}
\begin{proof} Applying $\nabla^k$ to \eqref{hr330}  and  multiplying it by $\nabla^k (\phi^2-1)$, we have
\begin{equation}\label{hh59}\begin{aligned}
		\f{1}{2}&\f{d}{dt}\|\nabla^{k}\lf(\phi^2-1\rg)\|^2+\int\f{\epsilon}{\rho^2}|\nabla^{k+1}\lf(\phi^2-1\rg)|^2d\mathbf{x}+\int\f{2}{\rho \epsilon}|\nabla^{k}\lf(\phi^2-1\rg)|^2d\mathbf{x}\\	&=\epsilon\underline{\sum_{1\leq l\leq k}
			C_{k}^{l}\int \nabla^{l}\lf(\f{1}{\rho^2}\rg) \nabla^{k-l}\Delta\lf(\phi^2-1\rg) \nabla^{k}\lf(\phi^2-1\rg) d\mathbf{x}}_{I_{16}}\\
		&-\epsilon\underline{\int \nabla\lf(\f{1}{\rho^2}\rg) \nabla^{k+1}\lf(\phi^2-1\rg) \nabla^{k}\lf(\phi^2-1\rg) d\mathbf{x}}_{I_{17}}
			\end{aligned}
\end{equation}

\begin{equation*}\label{hh59}\begin{aligned}
&-\f{2}{\epsilon}\underline{\sum_{1\leq l\leq k}
			C_{k}^{l}\int \nabla^{l}\lf(\f{1}{\rho}\rg) \nabla^{k-l} \lf(\phi^2-1\rg) \nabla^{k}\lf(\phi^2-1\rg) d\mathbf{x}}_{I_{18}}\\
		& -\underline{\int \nabla^k \lf[\mathbf{u}\cdot\nabla\lf(\phi^2-1\rg)\rg]  \nabla^k\lf(\phi^2-1\rg) d\mathbf{x}}_{I_{19}}-2\epsilon\underline{\int \nabla^k \lf(\f{|\nabla\phi|^2}{\rho^2}\rg)  \nabla^k\lf(\phi^2-1\rg) d\mathbf{x}}_{I_{20}}\\
			&-\f{2}{\epsilon}\underline{\sum_{0\leq l\leq k}
			C_{k}^{l}\int \nabla^{l}\lf(\f{\phi^2-1}{\rho}\rg) \nabla^{k-l} \lf(\phi^2-1\rg) \nabla^{k}\lf(\phi^2-1\rg) d\mathbf{x}}_{I_{21}}.
\end{aligned}
\end{equation*}
We estimate  $I_{i} (i=16, \cdots, 21)$. For  $I_{16}$, we have
$$\begin{aligned}
I_{16}&\lesssim \sum_{1\leq l\leq k} \|\nabla^{l}\lf(\f{1}{\rho^2}\rg) \nabla^{k-l}\Delta\lf(\phi^2-1\rg)\|_{L^{\f{6}{5}}}\|\nabla^{k}\lf(\phi^2-1\rg)\|_{L^6}\\&\overset{\eqref{h20}}\lesssim \sum_{1\leq l\leq k} \|\nabla^{l}\lf(\f{1}{\rho^2}\rg) \nabla^{k-l}\Delta\lf(\phi^2-1\rg)\|_{L^{\f{6}{5}}}\|\nabla^{k+1}\lf(\phi^2-1\rg)\|.
\end{aligned}
$$
If $l\leq \lf[\f{k}{2}\rg]$,   we get
$$
\begin{aligned}
&\|\nabla^{l}\lf(\f{1}{\rho^2}\rg) \nabla^{k-l}\Delta\lf(\phi^2-1\rg)\|_{L^{\f{6}{5}}}\overset{\eqref{hh20}}\lesssim \|\nabla^{l}\sigma\|_{L^3}\|\nabla^{k-l+2}\lf(\phi^2-1\rg)\| \\
&\overset{\eqref{h20}}\lesssim \|\nabla^\alpha\sigma\|^{1-\f{l-1}{k+1}}\|\nabla^{k+1}\sigma\|^{\f{l-1}{k+1}}\|\phi^2-1\|^{\f{l-1}{k+1}}\|\nabla^{k+1}\lf(\phi^2-1\rg)\|^{1-\f{l-1}{k+1}}
\\
&\overset{\eqref{h310}}\lesssim  M\lf(\|\nabla^{k+1}\sigma\|+\|\nabla^{k+1}\lf(\phi^2-1\rg)\|\rg),
\end{aligned}$$
where $\alpha$ is defined by
$$\begin{aligned}
\frac{l-1}{3}&=\lf(\frac{\alpha}{3}-\frac{1}{2}\rg)\lf(1-\f{l-1}{k+1}\rg)+\lf(\frac{k+1}{3}-\frac{1}{2}\rg)\f{l-1}{k+1}\\& \rightarrow \alpha=\f{3}{2}\lf(1+\f{l-1}{k+2-l}\rg)\in \lf[\f{3}{2}, 3\rg].
\end{aligned}
$$
If $\lf[\f{k}{2}\rg]+1\leq l\leq k$, then  we get
$$
\begin{aligned}
&\|\nabla^{l}\lf(\f{1}{\rho^2}\rg) \nabla^{k-l}\Delta\lf(\phi^2-1\rg)\|_{L^{\f{6}{5}}}\overset{\eqref{hh20}}\lesssim \|\nabla^{l}\sigma\|\|\nabla^{k-l+2}\lf(\phi^2-1\rg)\|_{L^3} \\
&\overset{\eqref{h20}}\lesssim \|\sigma\|^{1-\f{l}{k+1}}\|\nabla^{k+1}\sigma\|^{\f{l}{k+1}}\|\nabla^\alpha\lf(\phi^2-1\rg)\|^{\f{l}{k+1}}\|\nabla^{k+1}\lf(\phi^2-1\rg)\|^{1-\f{l}{k+1}}
\\
&\overset{\eqref{h310}}\lesssim  M\lf(\|\nabla^{k+1}\sigma\|+\|\nabla^{k+1}\lf(\phi^2-1\rg)\|\rg),
\end{aligned}$$
where $\alpha$ is defined by
\begin{eqnarray*}
&&\displaystyle\frac{k+1-l}{3}=\lf(\frac{\alpha}{3}-\frac{1}{2}\rg)\f{l}{k+1}+\lf(\frac{k+1}{3}-\frac{1}{2}\rg)\lf(1-\f{l}{k+1}\rg) \\
&&\displaystyle\Longrightarrow \alpha=\f{3(k+1)}{2l}\in \lf[\f{3}{2},3\rg].
\end{eqnarray*}
Therefore, we obtain
$$
|I_{16}|\lesssim  M\lf(\|\nabla^{k+1}\sigma\|^2+\|\nabla^{k+1}\lf(\phi^2-1\rg)\|^2\rg).$$
It is easy to check that
$$\begin{aligned}
I_{17}&\lesssim  \|\nabla\lf(\f{1}{\rho^2}\rg)\|_{L^3}\| \nabla^{k+1}\lf(\phi^2-1\rg)\|\|\nabla^{k}\lf(\phi^2-1\rg)\|_{L^6}\\&\overset{\eqref{h20}}\lesssim \|\nabla\sigma\|_{L^3}\|\nabla^{k+1}\lf(\phi^2-1\rg)\|^2\lesssim M\|\nabla^{k+1}\lf(\phi^2-1\rg)\|^2.
\end{aligned}
$$
By the similar arguments,  we have
$$
I_{18}\lesssim \sum_{1\leq l\leq k} \|\nabla^{l}\lf(\f{1}{\rho}\rg) \nabla^{k-l}\lf(\phi^2-1\rg)\|\|\nabla^{k}\lf(\phi^2-1\rg)\|.
$$
If $l\leq \lf[\f{k}{2}\rg]$,   we get
$$
\begin{aligned}
&\|\nabla^{l}\lf(\f{1}{\rho}\rg) \nabla^{k-l}\lf(\phi^2-1\rg)\|\overset{\eqref{hh20}}\lesssim \|\nabla^{l}\sigma\|_{L^3}\|\nabla^{k-l}\lf(\phi^2-1\rg)\|_{L^6} \\
&\overset{\eqref{h20}}\lesssim \|\nabla^\alpha\sigma\|^{1-\f{l}{k}}\|\nabla^{k+1}\sigma\|^{\f{l}{k}}\|\nabla\lf(\phi^2-1\rg)\|^{\f{l}{k}}\|\nabla^{k+1}\lf(\phi^2-1\rg)\|^{1-\f{l}{k}}
\\
&\overset{\eqref{h310}}\lesssim  M\lf(\|\nabla^{k+1}\sigma\|+\|\nabla^{k+1}\lf(\phi^2-1\rg)\|\rg),
\end{aligned}$$
where $\alpha$ is defined by
\begin{eqnarray}
\label{ggg-1}
&&\displaystyle\frac{l-1}{3}=\lf(\frac{\alpha}{3}-\frac{1}{2}\rg)\lf(1-\f{l}{k}\rg)+\lf(\frac{k+1}{3}-\frac{1}{2}\rg)\f{l}{k}\notag\\
&&\displaystyle\Longrightarrow \alpha=\f{1}{2}-\f{l}{2(k-l)}\in \lf[0,\f{1}{2}\rg].
\end{eqnarray}
If $\lf[\f{k}{2}\rg]+1\leq l\leq k$, then  we get
$$
\begin{aligned}
&\|\nabla^{l}\lf(\f{1}{\rho}\rg) \nabla^{k-l}\lf(\phi^2-1\rg)\|\overset{\eqref{hh20}}\lesssim \|\nabla^{l}\sigma\|_{L^{6}}\|\nabla^{k-l}\lf(\phi^2-1\rg)\|_{L^3} \\
&\overset{\eqref{h20}}\lesssim \|\nabla\sigma\|^{1-\f{l}{k}}\|\nabla^{k+1}\sigma\|^{\f{l}{k}}\|\nabla^\alpha\lf(\phi^2-1\rg)\|^{\f{l}{k}}\|\nabla^{k+1}\lf(\phi^2-1\rg)\|^{1-\f{l}{k}}
\\
&\overset{\eqref{h310}}\lesssim  M\lf(\|\nabla^{k+1}\sigma\|+\|\nabla^{k+1}\lf(\phi^2-1\rg)\|\rg),
\end{aligned}$$
where $\alpha$ is defined by
\begin{equation}
\label{ggg-2}
\frac{k-1-l}{3}=\lf(\frac{\alpha}{3}-\frac{1}{2}\rg)\f{l}{k}+\lf(\frac{k+1}{3}-\frac{1}{2}\rg)\lf(1-\f{l}{k}\rg)\Longrightarrow \alpha=1-\f{k}{2l}\in \lf[0, 1\rg].\end{equation}
Therefore, we obtain
$$
|I_{18}|\lesssim  M\lf(\|\nabla^{k+1}\sigma\|^2+\|\nabla^{k}\lf(\phi^2-1\rg)\|^2+\|\nabla^{k+1}\lf(\phi^2-1\rg)\|^2\rg).$$
Also, for  $I_{19}$, we have
$$\begin{aligned}
I_{19}&=\int \mathbf{u}\cdot \nabla^{k+1}\lf(\phi^2-1\rg)\nabla^k\lf(\phi^2-1\rg) d\mathbf{x}\\
&+\sum_{1\leq l\leq k}
C_{k}^{l}\int \nabla^l\mathbf{u}\cdot \nabla^{k+1-l}\lf(\phi^2-1\rg)\nabla^k\lf(\phi^2-1\rg) d\mathbf{x}\\& \lesssim  M\lf(\|\nabla^{k}\lf(\phi^2-1\rg)\|^2+\|\nabla^{k+1}\lf(\phi^2-1\rg)\|^2\rg)\\&
+\sum_{1\leq l\leq k} \|\nabla^{l}\mathbf{u}\cdot \nabla^{k-l+1}\lf(\phi^2-1\rg)\|\|\nabla^{k}\lf(\phi^2-1\rg)\|.
\end{aligned}
$$
If $l\leq \lf[\f{k+1}{2}\rg]$,   we get
$$
\begin{aligned}
&\|\nabla^{l}\mathbf{u}\cdot \nabla^{k+1-l}\lf(\phi^2-1\rg)\|\lesssim \|\nabla^{l}\mathbf{u}\|_{L^3}\|\nabla^{k+1-l}\lf(\phi^2-1\rg)\|_{L^6} \\
&\overset{\eqref{h20}}\lesssim \|\nabla^\alpha\mathbf{u}\|^{1-\f{l-1}{k}}\|\nabla^{k+1}\mathbf{u}\|^{\f{l-1}{k}}\|\nabla\lf(\phi^2-1\rg)\|^{\f{l-1}{k}}\|\nabla^{k+1}\lf(\phi^2-1\rg)\|^{1-\f{l-1}{k}}
\\
&\overset{\eqref{h310}}\lesssim  M\lf(\|\nabla^{k+1}\mathbf{u}\|+\|\nabla^{k+1}\lf(\phi^2-1\rg)\|\rg),
\end{aligned}$$
where $\alpha$ is defined by
\begin{eqnarray*}
&&\displaystyle\frac{l-1}{3}=\lf(\frac{\alpha}{3}-\frac{1}{2}\rg)\lf(1-\f{l-1}{k}\rg)+\lf(\frac{k+1}{3}-\frac{1}{2}\rg)\f{l-1}{k} \\
&&\displaystyle\Longrightarrow \alpha=\f{3}{2}+\f{l-1}{2(k+1-l)}\in \lf[\f{3}{2}, \f{5}{2}\rg].
  \end{eqnarray*}
If $\lf[\f{k+1}{2}\rg]+1\leq l\leq k$, then  we get
$$
\begin{aligned}
&\|\nabla^{l}\mathbf{u} \nabla^{k+1-l}\lf(\phi^2-1\rg)\|\lesssim \|\nabla^{l}\mathbf{u}\|_{L^{6}}\|\nabla^{k+1-l}\lf(\phi^2-1\rg)\|_{L^3} \\
&\overset{\eqref{h20}}\lesssim \|\nabla\mathbf{u}\|^{1-\f{l}{k}}\|\nabla^{k+1}\mathbf{u}\|^{\f{l}{k}}\|\nabla^\alpha\lf(\phi^2-1\rg)\|^{\f{l}{k}}\|\nabla^{k+1}\lf(\phi^2-1\rg)\|^{1-\f{l}{k}}
\\
&\overset{\eqref{h310}}\lesssim  M\lf(\|\nabla^{k+1}\mathbf{u}\|+\|\nabla^{k+1}\lf(\phi^2-1\rg)\|\rg),
\end{aligned}$$
where $\alpha$ is defined by
\begin{eqnarray*}
 &&\displaystyle \frac{k-l}{3}=\lf(\frac{\alpha}{3}-\frac{1}{2}\rg)\f{l}{k}+\lf(\frac{k+1}{3}-\frac{1}{2}\rg)\lf(1-\f{l}{k}\rg) \\
 &&\displaystyle\Longrightarrow \alpha=1+\f{k}{2l}\in \lf[1, 2\rg].
\end{eqnarray*}
Therefore, we obtain
$$
|I_{19}|\lesssim  M\lf(\|\nabla^{k+1}\mathbf{u}\|^2+\|\nabla^{k}\lf(\phi^2-1\rg)\|^2+\|\nabla^{k+1}\lf(\phi^2-1\rg)\|^2\rg).$$
Similarity,  we have
$$
\begin{aligned}
I_{20}&=\int \f{\nabla\phi}{\rho^2}\cdot \nabla^{k+2}\phi\nabla^k\lf(\phi^2-1\rg) d\mathbf{x}\\
&+\sum_{1\leq l\leq k}
C_{k}^{l}\int \nabla^l\lf(\f{\nabla\phi}{\rho^2}\rg)\cdot \nabla^{k+1-l}\phi\nabla^k\lf(\phi^2-1\rg) d\mathbf{x}\\& \lesssim  M\lf(\|\nabla^{k}\lf(\phi^2-1\rg)\|^2+\|\nabla^{k+1}\phi\|^2\rg)\\&
+\sum_{1\leq l\leq k} \|\nabla^{l}\lf(\f{\nabla\phi}{\rho^2}\rg)\cdot \nabla^{k-l+1}\phi\|\|\nabla^{k}\lf(\phi^2-1\rg)\|.
\end{aligned}
$$
If $l\leq \lf[\f{k+1}{2}\rg]$,   we get
$$
\begin{aligned}
&\|\nabla^{l}\lf(\f{\nabla\phi}{\rho^2}\rg)\cdot \nabla^{k+1-l}\phi\|\overset{\eqref{hh20}}\lesssim\lf(\|\nabla^{l+1}\phi\|_{L^3}+\|\nabla^{l}\sigma\|_{L^3}\rg) \|\nabla^{k+1-l}\phi\|_{L^6} \\
&\overset{\eqref{h20}}\lesssim \lf( \|\nabla^{\alpha+1}\phi\|^{1-\f{l}{k+1}}\|\nabla^{k+2}\phi\|^{\f{l}{k+1}}+\|\nabla^{\alpha}\sigma\|^{1-\f{l}{k+1}}\|\nabla^{k+1}\sigma\|^{\f{l}{k+1}}\rg)\|\nabla\phi\|^{\f{l}{k+1}}\|\nabla^{k+2}\phi\|^{1-\f{l}{k+1}}
\\
&\overset{\eqref{h310}}\lesssim  M\lf(\|\nabla^{k+1}\sigma\|+\|\nabla^{k+2}\phi\|\rg),
\end{aligned}$$
where $\alpha$ is defined by
\begin{eqnarray*}
&&\displaystyle  \frac{l-1}{3}=\lf(\frac{\alpha}{3}-\frac{1}{2}\rg)\lf(1-\f{l}{k+1}\rg)+\lf(\frac{k+1}{3}-\frac{1}{2}\rg)\f{l}{k+1} \\
&&\displaystyle \Longrightarrow \alpha=\f{k+1}{2(k+1-l)}\in \lf[0, 1\rg].
\end{eqnarray*}
If $\lf[\f{k+1}{2}\rg]+1\leq l\leq k$, then  we get
$$
\begin{aligned}
&\|\nabla^{l}\lf(\f{\nabla\phi}{\rho^2}\rg)\cdot \nabla^{k+1-l}\phi\|\overset{\eqref{hh20}}\lesssim\lf(\|\nabla^{l+1}\phi\|_{L^6}+\|\nabla^{l}\sigma\|_{L^6}\rg) \|\nabla^{k+1-l}\phi\|_{L^3} \\
&\overset{\eqref{h20}}\lesssim\lf( \|\nabla^2\phi\|^{1-\f{l}{k}}\|\nabla^{k+2}\phi\|^{\f{l}{k}}+ \|\nabla\sigma\|^{1-\f{l}{k}}\|\nabla^{k+1}\sigma|^{\f{l}{k}}\rg)\|\nabla^\alpha\phi\|^{\f{l}{k}}\|\nabla^{k+2}\phi\|^{1-\f{l}{k}}
\\
&\overset{\eqref{h310}}\lesssim  M\lf(\|\nabla^{k+1}\sigma\|+\|\nabla^{k+2}\phi\|\rg),
\end{aligned}$$
where $\alpha$ is defined by
$$
\frac{k-l}{3}=\lf(\frac{\alpha}{3}-\frac{1}{2}\rg)\f{l}{k}+\lf(\frac{k+2}{3}-\frac{1}{2}\rg)\lf(1-\f{l}{k}\rg)\rightarrow \alpha=2-\f{k}{2l}\in \lf[1, 2\rg].$$
Therefore, we obtain
$$
|I_{20}|\lesssim  M\lf(\|\nabla^{k+1}\sigma\|^2+\|\nabla^{k}\lf(\phi^2-1\rg)\|^2+\|\nabla^{k+2}\phi\|^2\rg).$$
Last,  we have
$$
\begin{aligned}
I_{21} \lesssim  M\|\nabla^{k}\lf(\phi^2-1\rg)\|^2
+\sum_{1\leq l\leq k} \|\nabla^{l}\lf(\f{\phi^2-1}{\rho}\rg) \nabla^{k-l} \lf(\phi^2-1\rg)\|\|\nabla^{k}\lf(\phi^2-1\rg)\|.
\end{aligned}
$$
If $l\leq \lf[\f{k}{2}\rg]$,   we get
$$
\begin{aligned}
&\|\nabla^{l}\lf(\f{\phi^2-1}{\rho}\rg) \nabla^{k-l}\lf(\phi^2-1\rg)\|\\
&\overset{\eqref{hh20}}\lesssim\lf(\|\nabla^{l}\lf(\phi^2-1\rg)\|_{L^3}+\|\nabla^{l}\sigma\|_{L^3}\rg) \|\nabla^{k-l}\lf(\phi^2-1\rg)\|_{L^6} \\
&\overset{\eqref{h20}}\lesssim\lf( \|\nabla^\alpha\lf(\phi^2-1\rg)\|^{1-\f{l}{k}}\|\nabla^{k+1}\lf(\phi^2-1\rg)\|^{\f{l}{k}}+ \|\nabla^\alpha\sigma\|^{1-\f{l}{k}}\|\nabla^{k+1}\sigma\|^{\f{l}{k}}\rg)\times\\
&\qquad\times\|\nabla\lf(\phi^2-1\rg)\|^{\f{l}{k}}\|\nabla^{k+1}\lf(\phi^2-1\rg)\|^{1-\f{l}{k}}
\\
&\overset{\eqref{h310}}\lesssim  M\lf(\|\nabla^{k+1}\sigma\|+\|\nabla^{k+1}\lf(\phi^2-1\rg)\|\rg),
\end{aligned}$$
where $\alpha$ is defined by \eqref{ggg-1}.
If $\lf[\f{k}{2}\rg]+1\leq l\leq k$, then  we get
$$
\begin{aligned}
&\displaystyle\|\nabla^{l}\lf(\f{\phi^2-1}{\rho}\rg) \nabla^{k-l}\lf(\phi^2-1\rg)\|\\
&\displaystyle\overset{\eqref{hh20}}\lesssim\lf(\|\nabla^{l}\lf(\phi^2-1\rg)\|_{L^{6}}+\|\nabla^{l}\sigma\|_{L^{6}}\rg) \|\nabla^{k-l}\lf(\phi^2-1\rg)\|_{L^3} \\
&\displaystyle\overset{\eqref{h20}}\lesssim\lf( \|\nabla\lf(\phi^2-1\rg)\|^{1-\f{l}{k}}\|\nabla^{k+1}\lf(\phi^2-1\rg)\|^{\f{l}{k}}+ \|\nabla\sigma\|^{1-\f{l}{k}}\|\nabla^{k+1}\sigma\|^{\f{l}{k}}\rg)\times\\
&\displaystyle\times\|\nabla^\alpha\lf(\phi^2-1\rg)\|^{\f{l}{k}}\|\nabla^{k+1}\lf(\phi^2-1\rg)\|^{1-\f{l}{k}}
\\
&\overset{\eqref{h310}}\lesssim  M\lf(\|\nabla^{k+1}\sigma\|+\|\nabla^{k+1}\lf(\phi^2-1\rg)\|\rg),
\end{aligned}$$
where $\alpha$ is defined by \eqref{ggg-2}.
Therefore, we obtain
$$
|I_{21}|\lesssim  M\lf(\|\nabla^{k+1}\sigma\|^2+\|\nabla^{k}\lf(\phi^2-1\rg)\|^2+\|\nabla^{k+2}\phi\|^2\rg).$$
Substituting the estimates on $I_{i} (i=16, \cdots, 21)$  into \eqref{hh59} and  using \eqref{h315},  we obtain \eqref{hh58}.  The proof of Lemma \ref{lem51} is completed.\end{proof}

\vspace{0.2cm}We will continue the  proof of \eqref{hg124} for $s\in (0, \f{1}{2}]$. Summing up the estimate \eqref{hh58} of Lemma \ref{lem51} for from $k=l$ to $2$, we obtain
\begin{equation}
\label{hgh513}\begin{aligned}
&\displaystyle\f{1}{2}\f{d}{dt}\|\nabla^{l} \lf(\phi^2-1\rg)\|_{H^{2-l}}^2+\f{\epsilon}{8\bar{\rho}^2}\|\nabla^{l+1} \lf(\phi^2-1\rg)\|_{H^{2-l}}^2\\
&	\displaystyle\lesssim M\lf(\|\nabla ^{l+1} \sigma\|^2_{H^{2-l}}+ \|\nabla^{l+1}\mathbf{u}\|^2_{H^{3-l}}+ \|\nabla^{l+2} \phi\|^2_{H^{2-l}}\rg)
\end{aligned}
\end{equation}for $l=1,2$.
Therefore, after  summing   \eqref{h54} and  \eqref{hgh513}, we have
\begin{equation}
\label{hg54}
\f{d}{dt}\mathcal{F}_l(t)+G_{l}(t)\leq 0,
\end{equation} where
\begin{equation}\label{hg55}\begin{aligned}&\displaystyle
\mathcal{F}_l(t):=\mathcal{E}_l(t)+\f{1}{2}\|\nabla^{l} \lf(\phi^2-1\rg)\|_{H^{2-l}}^2
\\
&\displaystyle\begin{aligned}
G_{l}(t):&=\Lambda_{l}(t)+\f{\epsilon}{8\bar{\rho}^2}\|\nabla^{l+1} \lf(\phi^2-1\rg)\|_{H^{2-l}}^2\\
&\displaystyle-C_0M\lf(\|\nabla ^{l+1} \sigma\|^2_{H^{2-l}}+ \|\nabla^{l+1}\mathbf{u}\|^2_{H^{3-l}}+ \|\nabla^{l+2} \phi\|^2_{H^{2-l}}\rg).\end{aligned}\end{aligned}
\end{equation}
Noticing that  $M>0$ is  small, and using \eqref{h55} and \eqref{hh55}, we obtain from \eqref{hg55} that
\begin{equation}\label{hg56}\begin{aligned}
&\displaystyle\mathcal{F}_l(t) \backsimeq \|\nabla^l(\sigma, \mathbf{u})(t)\|^2_{H^{3-l}}+\|\nabla^{l} (\nabla\phi, \phi^2-1)(t)\|^2_{H^{2-l}},\\
&\displaystyle  G_{l}(t)\backsimeq \|\nabla ^{l+1}(\sigma, \nabla\phi, \phi^2-1) (t)\|^2_{H^{2-l}}+ \|\nabla^{l+1}\mathbf{u}(t)\|^2_{H^{3-l}}
\end{aligned} \end{equation} uniformly for all $t\geq 0$.
If $l=0,1,2$, we may use \eqref{hg24} to have
\begin{equation}\label{h512}
\|\nabla^{l+1}f\|\gtrsim\|\Lambda^{-s}f\|^{-\f{1}{l+s}}\|\nabla^{l}f\|^{1+\f{1}{l+s}}.\end{equation}
By \eqref{h512} and \eqref{h124}, we get
$$\|\nabla^{l+1}(\sigma, \mathbf{u}, \nabla\phi, \phi^2-1)\|\gtrsim \|\nabla^{l}(\sigma, \mathbf{u}, \nabla\phi, \phi^2-1)\|^{1+\f{1}{l+s}}, $$
which implies
\begin{equation}\label{h513}
	\begin{aligned}
		&\|\nabla ^{l+1} (\sigma,  \nabla\phi, \phi^2-1)(t)\|^2_{H^{N-l-1}}+ \|\nabla^{l+1}\mathbf{u}(t)\|^2_{H^{3-l}}\\&
		\gtrsim \lf(\|\nabla ^{l} (\sigma, \nabla\phi, \phi^2-1)(t)\|^2_{H^{N-l-1}}+ \|\nabla^{l}\mathbf{u}(t)\|^2_{H^{3-l}}\rg)^{1+\f{1}{l+s}}.
	\end{aligned}
\end{equation} Also, using $\|\sigma(t)\|_{H^3}\lesssim 1$ due to \eqref{h18}, we have
\begin{equation}\label{hh513}
\|\nabla^3\sigma\|\gtrsim\|\nabla^3\sigma\|^{1+\f{1}{l+s}}.\end{equation}
By \eqref{h513} and \eqref{hh513}, we  obtain from \eqref{h56} that
\begin{equation}\label{hg513}
G_{l}(t)\gtrsim\lf(\mathcal{F}_l(t)\rg)^{1+\f{1}{l+s}}.\end{equation}
Using  \eqref{hg513},  we deduce from \eqref{hg54} with  $m=3$ that
$$\f{d}{dt}\mathcal{F}_l(t)+C_0\lf(\mathcal{F}_l(t)\rg)^{1+\f{1}{l+s}}\leq 0. $$
Solving this inequality directly gives  (see \eqref{h119})
$$\mathcal{F}_l(t)\lesssim (1+t)^{-(l+s)}, $$
which implies  \eqref{hg124} for $s\in \lf[0, \f{1}{2}\rg]$,  that is,
$$
\|\nabla^l(\sigma, \mathbf{u})(t)\|^2_{H^{3-l}}+\|\nabla^{l} (\nabla\phi, \phi^2-1)(t)\|^2_{H^{2-l}}\lesssim (1+t)^{-(l+s)},
$$due to $\eqref{hg56}_1$.

\vspace{0.2cm} Next,  we prove \eqref{h124}-\eqref{hg124} for  $s\in \lf(\f{1}{2}, \f{3}{2}\rg)$. Notice that the arguments for the case $s\in \lf[0, \f{1}{2}\rg]$ can not
be applied to this case. However, observing that we have $(\sigma_0, \mathbf{u}_0, \nabla\phi_0)\in \dot{H}^{-\f{1}{2}}(\mathbb{R}^3)$ since $\dot{H}^{-s}(\mathbb{R}^3)\cap L^2(\mathbb{R}^3)\subset \dot{H}^{-s'}(\mathbb{R}^3)$ for any $s'\in[0,s]$,  we then deduce from what we have proved for \eqref{hg124} with $s=\f{1}{2}$ that the
following decay result holds:
\begin{equation}
	\label{h514} \|\nabla^l(\sigma, \mathbf{u})(t)\|^2_{H^{3-l}}+\|\nabla^{l} (\nabla\phi, \phi^2-1)(t)\|^2_{H^{2-l}}\lesssim (1+t)^{-(l+\f{1}{2})},
\end{equation}  for $l=0,1,2$.
Hence, by \eqref{h514}, we deduce from \eqref{hr42}  and \eqref{h18} that for $s\in (\f{1}{2}, \f{3}{2})$,
$$\begin{aligned}
&\mathcal{E}_{-s}(t)\lesssim \mathcal{E}_{-s}(0)+\\
& \ +\int_0^t \|(\sigma, \mathbf{u}, \phi^2-1, \nabla\phi)\|^{s-\f{1}{2}}\lf(
\|\nabla(\sigma, \mathbf{u},  \nabla\phi)\|_{H^1}+\|\nabla(\phi^2-1)\|\rg)^{\f{5}{2}-s}\sqrt{\mathcal{E}_{-s}(\tau)}d\tau\\
&\lesssim 1+\int_0^t (1+\tau)^{-\alpha}\sqrt{\mathcal{E}_{-s}(\tau)}d\tau
\lesssim 1+ \sup_{0\leq \tau\leq t}\sqrt{\mathcal{E}_{-s}(\tau)},
\end{aligned}$$where $\alpha$ is defined by
$$\alpha=\f{1}{4}\lf(s-\f{1}{2}\rg)+\f{1}{2}\lf(1+\f{1}{2}\rg)\lf(\f{5}{2}-s\rg)>1,\quad \text{since}\,\,\, s<\f{3}{2}.$$
This implies  \eqref{h124} for  $s\in (\f{1}{2}, \f{3}{2})$, that is
\begin{equation}
	\label{h515}
	\Big\|(\rho-\bar{\rho}, \mathbf{u}, \nabla\phi, \phi^2-1)(t)\Big\|_{\dot{H}^{-s}}\lesssim 1, \quad \text{for} \,\,s\in \lf(\f{1}{2}, \f{3}{2}\rg).
\end{equation}
Now that we have proved \eqref{h515}, we may repeat the arguments leading to \eqref{hg124} for  $s\in [0, \f{1}{2}]$ to prove that they hold also for  $s\in (\f{1}{2}, \f{3}{2})$.
The proof of \eqref{h124}-\eqref{hg124} for  $s\in [0, \f{3}{2})$ is completed.

\section*{Conflict of Interests}
The authors declare that there is no conflict of interest regarding the publication of this paper.

\end{document}